\documentclass[12pt,twoside,reqno]{amsart}
\usepackage{amsmath}
\usepackage{amsfonts}
\usepackage{amssymb}
\usepackage{enumitem}
\usepackage{xcolor}
\usepackage{mathrsfs}
\usepackage{cite}
\newtheorem{theorem}{Theorem}[section]
\newtheorem{corollary}[theorem]{Corollary}
\newtheorem{lemma}[theorem]{Lemma}
\newtheorem{proposition}[theorem]{Proposition}

\newtheorem{remark}[theorem]{Remark}
\numberwithin{equation}{section}
\begin{document}
\title{Stationary solutions to\\ coagulation-fragmentation equations} 
\author{Philippe Lauren\c{c}ot}
\address{Institut de Math\'ematiques de Toulouse, UMR~5219, Universit\'e de Toulouse, CNRS \\ F--31062 Toulouse Cedex 9, France}
\email{laurenco@math.univ-toulouse.fr}


\date{\today}

\begin{abstract}
Existence of stationary solutions to the coagulation-fragmentation equation is shown when the coagulation kernel $K$ and the overall fragmentation rate $a$ are given by $K(x,y) = x^\alpha y^\beta + x^\beta x^\alpha$ and $a(x)=x^\gamma$, respectively, with $0\le \alpha\le \beta \le 1$, $\alpha+\beta\in [0,1)$, and $\gamma> 0$. The proof requires two steps: a dynamical approach is first used to construct stationary solutions under the additional assumption that the coagulation kernel and the overall fragmentation rate are bounded from below by a positive constant. The general case is then handled by a compactness argument.
\end{abstract}

\maketitle

%
%
\pagestyle{myheadings}
\markboth{\sc{Philippe Lauren\c cot}}{\sc{Stationary solutions to coagulation-fragmentation equations}}

\section{Introduction}\label{sec1}

The coagulation-fragmentation equation is a mean-field model describing the time evolution of the size distribution function $f$ of a system of particles increasing their size by pairwise merging or reducing it by splitting, no matter being loss during these processes. Denoting the coagulation kernel, the overall fragmentation rate, and the daughter distribution function by $K$, $a$, and $b$, respectively, the coagulation-fragmentation equation reads
\begin{subequations}\label{a1}
\begin{align}
\partial_t f & = \mathcal{C}f + \mathcal{F}f\ , \qquad (t,x) \in (0,\infty)^2\ , \label{a1a} \\
f(0) & = f^{in}\ , \qquad x\in (0,\infty)\ , \label{a1ai}
\end{align}
where the coagulation term $\mathcal{C}f$ is given by
\begin{equation}
\mathcal{C}f(x) := \frac{1}{2} \int_0^x K(y,x-y) f(x-y) f(y)\ \mathrm{d}y - \int_0^\infty K(x,y) f(x) f(y)\ \mathrm{d}y\ , \qquad x>0\ , \label{a1b}
\end{equation}
and the fragmentation term $\mathcal{F}f$ by
\begin{equation}
\mathcal{F}f(x) := - a(x) f(x) + \int_x^\infty a(y) b(x,y) f(y)\ \mathrm{d}y\ , \qquad x>0\ . \label{a1c}
\end{equation}
\end{subequations}
The first term in \eqref{a1b} accounts for the formation of particles of size $x>0$ as a consequence of the merging of two smaller particles with respective sizes $y\in (0,x)$ and $x-y$. The second term in \eqref{a1b} and the first term in \eqref{a1c} describe the depletion of particles of size $x>0$ due to coalescence with other particles and fragmentation, respectively. Finally, the breakup of a particle of size $y>x$ produces fragments of various sizes ranging in $(0,y)$, including fragments of size $x$ according to the distribution $b(x,y)$ as indicated by the second term in \eqref{a1c}. We further assume that there is no loss of matter during the breakage process, which amounts to require that $b$ satisfies
\begin{equation}
\int_0^y x b(x,y)\ \mathrm{d}x = y \ , \qquad y>0\ , \;\text{ and }\; b(x,y)=0\ , \qquad x>y>0\ . \label{a2}
\end{equation}
Since there is also no loss of matter during coalescence, the total mass of the system is expected to be invariant throughout time evolution; that is, 
\begin{equation}
\int_0^\infty x f(t,x)\ \mathrm{d}x = \int_0^\infty x f(0,x)\ \mathrm{d}x\ , \qquad t\ge 0\ . \label{a3}
\end{equation}
Though this property may fail to be true when, either the coagulation is too strong compared to the fragmentation, a phenomenon known as \textsl{gelation}, or the overall fragmentation rate $a$ is unbounded as $x\to 0$, a phenomenon known as \textsl{shattering}, both are excluded in the forthcoming analysis and we refer to \cite{EMP02, ELMP03, Jeon98, Laur00, Leyv83, LeTs81} and \cite{ArBa04, Fili61, McZi87}, respectively, for detailed information on these issues.

Our interest in this paper is rather related to the possible balance between coagulation and fragmentation, which are competing mechanisms. Indeed, the latter increases the number of particles and reduces the mean size of particles, while the former acts in the opposite direction. It is then of interest to figure out the outcome of this competition and, in particular, whether it could lead to stationary solutions. This is the issue we aim at investigating herein. 

The first example of coagulation-fragmentation equation featuring steady state solutions is the case of constant coefficients \cite{AiBa79}
\begin{equation}
\begin{split}
\partial_t f(t,x) & = \int_0^x \left[ f(t,x-y) f(t,y) - A_0 f(t,x) \right]\ \mathrm{d}y \\
& \quad - 2 \int_0^\infty \left[ f(t,x) f(t,y) - A_0 f(t,x+y) \right]\ \mathrm{d}y \ , \qquad (t,x)\in (0,\infty)^2\ ,
\end{split} \label{AB}
\end{equation}
which is obtained with the choice
\begin{equation}
K(x,y) = 2\ , \qquad a(x) = A_0 x\ , \qquad b(x,y) = \frac{2}{y}\ , \qquad 0<x<y\ , \label{a5}
\end{equation}
in \eqref{a1}. For any $z>0$, the function $Q_z$ defined by $Q_z(x) := A_0 e^{x\ln{z}}$, $x>0$, is a stationary solution to \eqref{AB} and $Q_z$ has finite total mass if and only if $z\in (0,1)$. The example \eqref{a5} is actually a particular case of coagulation and fragmentation coefficients satisfying the so-called \textit{detailed balance condition}: there are a non-negative symmetric function $F$ defined on $(0,\infty)^2$ and a non-negative function $Q$ defined on $(0,\infty)$ such that
\begin{subequations}\label{a6}
\begin{equation}
a(x) = \frac{1}{2} \int_0^x F(x_*,x-x_*)\ \mathrm{d}x_*\ , \qquad a(y) b(x,y) = F(x,y-x)\ , \qquad 0<x<y\ , \label{a6a}
\end{equation}
\begin{equation}
K(x,y) Q(x) Q(y) = F(x,y) Q(x+y)\ , \qquad (x,y)\in (0,\infty)^2\ . \label{a6b}
\end{equation}
\end{subequations}
Note that we recover \eqref{a5} from \eqref{a6} by setting $F\equiv 2 A_0$ and $Q\equiv A_0$. Thanks to \eqref{a6}, the equation \eqref{a1} reads
\begin{equation}
\begin{split}
\partial_t f(t,x) & = \frac{1}{2} \int_0^x \left[ K(x-y,y) f(t,x-y) f(t,y) - F(y,x-y) f(t,x) \right]\ \mathrm{d}y \\
& \quad - \int_0^\infty \left[ K(x,y) f(t,x) f(t,y) - F(x,y) f(t,x+y) \right]\ \mathrm{d}y \ , \qquad (t,x)\in (0,\infty)^2\ ,
\end{split} \label{DB}
\end{equation}
and $Q_z: x\mapsto Q(x) e^{x\ln{z}}$ is a stationary solution to \eqref{DB} for all $z\in (0,\infty)$. Whether $Q_z$ has finite total mass then depends on both the value of $z$ and the integrability properties of $Q$. We refer to \cite{Carr92, CadC94, LaMi02c, LaMi03} for a more detailed account on the various situations that may happen.

Coagulation and fragmentation coefficients satisfying the detailed balance condition \eqref{a6} are however far from being generic and different approaches have to be designed to investigate the existence of stationary solutions to \eqref{a1} when \eqref{a6} fails to hold. When the coagulation and fragmentation coefficients are given by 
\begin{equation}
K(x,y) = k_0 + k_1(x+y)\ , \qquad a(x) = A_0 x\ , \qquad b(x,y) = \frac{2}{y}\ , \qquad 0<x<y\ , \label{a7}
\end{equation}
the existence of a stationary solution to \eqref{a1} having total mass $\varrho>0$ is proved in \cite{DuSt96a} for all $\varrho>0$, the proof relying on a fixed point argument performed on the stationary version of \eqref{a1a}. It uses in an essential way the specific form of the coefficients and does not seem to extend to handle more general cases. Uniqueness and local stability of steady states are also established in \cite{DuSt96a}. In the same vein but with a completely different approach, a complete description of stationary solutions to \eqref{a1} is obtained in \cite[Theorem~5.1 \& Remark~5.2]{DLP17} when
\begin{equation}
K(x,y) = k_0 (xy)^{\lambda/2}\ , \qquad a(x) = A_0 x^{\lambda/2}\ , \qquad b(x,y) = \frac{2}{y}\ , \qquad 0<x<y\ , \label{a8}
\end{equation}
for some $\lambda\in [0,2]$, $k_0>0$, and $A_0>0$. Two steps are needed to obtain this result: first, when $\lambda=0$, $k_0=2$, and $A_0=1$, given an integrable stationary solution $f$ to \eqref{a1}, its Bernstein transform
$$
U(s) := \int_0^\infty \left( 1 - e^{-sx} \right) f(x)\ \mathrm{d}x\ , \qquad s\ge 0\ , 
$$
solves the integro-differential equation 
\begin{equation}
U(s)^2 + U(s) = \frac{2}{s} \int_0^s U(r)\ \mathrm{d}r\ , \qquad s>0\ , \qquad U(0) = 0\ . \label{a9}
\end{equation}
This equation turns out to have an explicit solution $U_\star$ which is the Bernstein transform of a non-negative function $f_\star\in L^1((0,\infty),(1+x)\mathrm{d}x)$ satisfying 
\begin{equation}
\int_0^\infty f_\star(x)\ \mathrm{d}x = \int_0^\infty x f_\star(x)\ \mathrm{d}x = 1\ , \label{a10}
\end{equation}
and any solution $U$ to \eqref{a9} is a dilation of $U_\star$; that is, there is $\mu>0$ such that $U(s) = U_\star(\mu s)$ for $s\ge 0$. Moreover, 
\begin{equation}
f_\star(x) \mathop{\sim}_{x\to 0} \frac{x^{-2/3}}{\Gamma(1/3)} \;\;\text{ and }\;\; f_\star(x) \mathop{\sim}_{x\to\infty} \frac{9}{8} \frac{x^{-3/2}}{\Gamma(1/2)} e^{-4x/27}\ .
\label{a11}
\end{equation}
In particular, $f_\star$ features an integrable singularity as $x\to 0$. To handle the case $\lambda>0$ in \eqref{a8}, it suffices to note that, if $f$ is a stationary solution to \eqref{a1} corresponding to coagulation and fragmentation rates given by \eqref{a7} for some $\lambda\in [0,2]$, $k_0>0$, and $A_0>0$, then $x\mapsto k_0 x^{\lambda/2} f(x)/2A_0$ is a stationary solution to \eqref{a1} corresponding to coagulation and fragmentation rates given by \eqref{a7} with $\lambda=0$, $k_0=2$, and $A_0=1$. Consequently, there is $\mu>0$ such that 
\begin{equation}
f(x) = \frac{2 A_0 \mu}{k_0} x^{-\lambda/2} f_\star(\mu x)\ , \qquad x\in (0,\infty)\ . \label{a12}
\end{equation} 
It readily follows from \eqref{a11} and \eqref{a12} that $f$ also features a singularity as $x\to 0$ which is not integrable if $\lambda>2/3$. However, the total mass of $f$ is finite for all $\lambda\in [0,2]$. Stability of stationary solutions is also investigated in \cite{DLP17} when $\lambda=0$, $k_0=2$, and $A_0=1$. 

The just described results only deal with very specific coagulation and fragmentation coefficients, and the approaches used in both cases exploit their particular structure. They are thus rather unlikely to extend to a wider setting. As far as we know, the only result handling a fairly general class of coagulation and fragmentation coefficients is to be found in \cite{EMRR05}, the coagulation and fragmentation coefficients being given by
\begin{subequations}\label{a13}
\begin{equation}
K(x,y) = x^{-\alpha} y^\beta + x^\beta y^{-\alpha}\ , \qquad (x,y)\in (0,\infty)^2\ , \label{a13a}
\end{equation}
and 
\begin{equation}
a(x) = a_0 x^\gamma\ , \qquad b(x,y) = \frac{1}{y} B\left( \frac{x}{y} \right)\ , \qquad 0<x<y\ , \label{a13b}
\end{equation}
where
\begin{equation}
(\alpha,\beta) \in [0,1]^2\ , \qquad \beta-\alpha \in [0,1)\ , \qquad \gamma\ge 0\ , \qquad a_0>0\ ,\label{a13c}
\end{equation} 
and 
\begin{equation}
B \;\text{ is a non-negative function in }\; L^1((0,1),(z+z^{-2\alpha}) \mathrm{d}z)\ .\label{a13d}
\end{equation}
\end{subequations}
Assuming further that $(\beta,\gamma)\ne (1,0)$ and $(\alpha,\gamma)\ne (0,0)$, the existence of a non-negative stationary solution to \eqref{a1} with total mass $\varrho$ is shown in \cite[Theorem~4.1]{EMRR05} for all $\varrho>0$. Furthermore, this stationary solution belongs to $L^1((0,\infty),x^m\mathrm{d}x)$ for all $m\ge -2\alpha$ and, under the additional assumption that $B\in L^\infty(0,1)$, it belongs to $L^p(0,\infty)$ for all $p\in [1,\infty)$. The approach developed to prove this result is of a completely different nature and actually relies on a dynamical approach. Roughly speaking, the basic idea is to find a suitable functional setting in which the initial value problem \eqref{a1} is well-posed, along with a closed and convex set $\mathcal{Z}$ which is compact for the associated topology and is positively invariant for the dynamical system associated to \eqref{a1} (in the sense that $f(t)\in\mathcal{Z}$ for all $t>0$ as soon as $f(0)\in\mathcal{Z}$). If a fixed point theorem is available in this functional setting, then a classical argument guarantees the existence of at least one stationary solution, see \cite[Theorem~16.5]{Am90}, \cite[Proof of Theorem~5.2]{GPV04}, and \cite[Theorem~1.2]{EMRR05}, for instance. Though this method merely gives the existence of a steady state solution without any information on uniqueness or stability, it is far more flexible than the previous ones and we shall partially employ it in the forthcoming analysis. Let us mention that it is also the cornerstone of the construction of mass-conserving self-similar solutions to the coagulation equation \cite{EMRR05, FoLa05, NiVe13b}. 

According to the previous description, no result on the existence of steady state solutions seems to be available for the classical coagulation kernel 
\begin{subequations}\label{a14}
\begin{equation} 
 K(x,y) = K_0 \left( x^\alpha y^\beta + x^\beta y^\alpha \right)\ , \qquad (x,y)\in (0,\infty)^2\ , \label{a14a}
\end{equation} 
with
\begin{equation}
0 \le \alpha \le \beta \le 1\ , \qquad \lambda := \alpha + \beta \in [0,1)\ , \label{a14b}
\end{equation}
\end{subequations}
and the purpose of this paper is to fill this gap for a rather large class of fragmentation coefficients. More precisely, we assume that there are
\begin{subequations}\label{a15}
\begin{equation} 
\gamma> 0\ , \qquad a_0>0\ , \qquad p_0>1\ , \label{a15a}
\end{equation}
and a non-negative function 
\begin{equation}
B\in L^1((0,1),z\mathrm{d}z)\cap L^{p_0}(0,1)\ , \qquad \int_0^1 z B(z)\ \mathrm{d}z = 1\ , \label{a15c}
\end{equation}
such that
\begin{equation}
a(x) = a_0 x^\gamma\ , \qquad b(x,y) = \frac{1}{y} B\left( \frac{x}{y} \right)\ , \qquad 0<x<y\ . \label{a15b}
\end{equation}
\end{subequations}
Note that the class of coagulation kernels \eqref{a14} includes the sum kernels corresponding to $\alpha=0$ and $\beta=\lambda\in [0,1)$ and the product kernels corresponding to $\alpha=\beta=\lambda/2\in [0,1/2)$. The constraint on $B$ in \eqref{a15c} stems from the conservation of matter \eqref{a2} during fragmentation events. Examples of daughter distribution functions satisfying \eqref{a15c} include the power-law breakup distribution
\begin{equation}
B(z) = B_{1,\nu}(z) := (\nu+2) z^\nu\ , \qquad z\in (0,1)\ , \qquad \nu>-1\ , \label{a15d}
\end{equation} 
and the parabolic breakup distribution
\begin{equation}
B(z) = B_{2,\nu}(z) := (\nu+2)(\nu+1) z^{\nu-1} (1-z)\ , \qquad z\in (0,1)\ , \qquad \nu>0\ . \label{a15e}
\end{equation}
Indeed, $B_{1,\nu}$ given by \eqref{a15d} satisfies \eqref{a15c} for any $p_0>1$ when $\nu\ge 0$ and for any $p_0\in (1,1/|\nu|)$ when $\nu\in (-1,0)$. Similarly, $B_{2,\nu}$ given by \eqref{a15e} satisfies \eqref{a15c} for any $p_0>1$ when $\nu\ge 1$ and $p_0\in (1,1/(1-\nu))$ when $\nu\in (0,1)$. 

Before stating the main result, let us introduce some notation. Throughout the paper, for $m\in\mathbb{R}$, we set
\begin{equation}
X_m := L^1((0,\infty),x^m\mathrm{d}x)\ , \qquad M_m(h) := \int_0^\infty x^m h(x)\ \mathrm{d}x\ , \quad h\in X_m\ , \label{a16}
\end{equation}
and denote the positive cone of $X_m$ by $X_m^+$. We also denote the space $X_m$ endowed with its weak topology by $X_{m,w}$. 

\begin{theorem}\label{thm1}
Assume that the coagulation and fragmentation coefficients satisfy \eqref{a14} and \eqref{a15}. Given $\varrho>0$ there exists at least a stationary (weak) solution $\varphi \in X_1^+$ to \eqref{a1} with the following properties:
\begin{itemize}
	\item [\textbf{(s1)}] $M_1(\varphi)=\varrho$;
	\item [\textbf{(s2)}] there are $p_1\in (1,p_0)$ and $m_1\in (\lambda,1)$ such that 
	\begin{equation*}
	\varphi \in L^{p_1}((0,\infty),x^{m_1+\gamma}\mathrm{d}x) \cap \bigcap_{m>\lambda} X_m\ ;
	\end{equation*}
	\item[\textbf{(s3)}] for all $\vartheta\in \Theta_1 := \{ h\in W^{1,\infty}(0,\infty)\ :\ h(0)=0\}$, 
\begin{align*}
& \frac{1}{2} \int_0^\infty \int_0^\infty K(x,y) \left[ \vartheta(x+y) - \vartheta(x) - \vartheta(y) \right] \varphi(x)  \varphi(y)\ \mathrm{d}y\mathrm{d}x \\
& \qquad = \int_0^\infty a(y)  \varphi(y) \left[ \vartheta(y) - \int_0^y \vartheta(x) b(x,y)\ \mathrm{d}x \right]\ \mathrm{d}y\ . 
\end{align*}
\end{itemize}
\end{theorem}

It is worth pointing out here that Theorem~\ref{thm1}~\textbf{(s2)} does not exclude a non-integrable singularity of $\varphi$ as $x\to 0$, a situation which may indeed occur, as we shall see below. This feature is not encountered for the coagulation and fragmentation coefficients given by \eqref{a13} and considered in \cite{EMRR05} when $\alpha<0$, as the unboundedness of the coagulation kernel for small sizes implies the vanishing of the stationary solution as $x\to 0$. This possible singular behaviour for small sizes is actually the main difficulty to be overcome in the analysis carried out below and requires a more involved approach, which we describe now.

The proof of Theorem~\ref{thm1} is carried out in two steps. We fix $\varrho>0$. Using the dynamical approach already alluded to, given $\varepsilon\in (0,1)$, we first construct a stationary solution $\varphi_{\varepsilon}\in X_1$ to 
\begin{equation}
\begin{split}
\partial_t f & = \mathcal{C}_\varepsilon f + \mathcal{F}_\varepsilon f\ , \qquad (t,x)\in (0,\infty)^2\ , \\
f(0) & = f^{in}\ , \qquad x\in (0,\infty)\ , 
\end{split} \label{CFe}
\end{equation}
satisfying $M_1(\varphi_{\varepsilon})=\varrho$, where the coagulation and fragmentation operators $\mathcal{C}_\varepsilon$ and $\mathcal{F}_\varepsilon$ are given by \eqref{a1b} with $K_\varepsilon := K+2 \varepsilon K_0$ instead of $K$ and \eqref{a1c} with $a_\varepsilon := a + a_0 \varepsilon^2$ instead of $a$, respectively. For this choice of coagulation and fragmentation coefficients, we actually build a closed convex and sequentially weakly compact subset $\mathcal{Z}_\varepsilon$ of $X_1$ such that solutions to \eqref{CFe} starting from an initial condition in $\mathcal{Z}_\varepsilon$ remain in $\mathcal{Z}_\varepsilon$ for all positive times. Recalling that, according to the Dunford-Pettis theorem, sequential weak compactness in $X_1$ requires to prevent concentration and escape of mass for small and large sizes, finding $\mathcal{Z}_\varepsilon$ amounts to derive time-independent estimates in $X_{m_0}\cap X_m\cap L^{p_2}(0,\infty)$ for some suitably chosen $m_0<1<m$ and $p_2>1$. While some of the moment estimates can be obtained directly for $\varepsilon=0$ (Section~\ref{sec2.1}), it does not seem to be possible to derive uniform integrability estimates without the positive lower bounds on $K_\varepsilon$ and $a_\varepsilon$ (Section~\ref{sec2.3}). Besides the construction of $\mathcal{Z}_\varepsilon$ (Section~\ref{sec3.4}), we also show the well-posedness of \eqref{CFe} in Section~\ref{sec3.1}, as well as the continuous dependence of solutions to \eqref{CFe} in $X_{1,w}$ with respect to the initial condition (Section~\ref{sec3.3}). To justify rigorously the computations performed in Section~\ref{sec2}, an additional approximation is needed and we shall actually work with truncated versions of $K_\varepsilon$ and $a_\varepsilon$. Thanks to this analysis, it remains to apply \cite[Theorem~1.2]{EMRR05} to obtain the existence of a stationary solution $\varphi_\varepsilon \in \mathcal{Z}_\varepsilon$ to \eqref{CFe} (Section~\ref{sec3.5}). To complete the proof of Theorem~\ref{thm1}, we are left with taking the limit $\varepsilon\to 0$. To this end, we realize that, since we have payed special attention to the dependence on $\varepsilon$ of the estimates derived in Section~\ref{sec2}, there is a sequentially weakly compact subset $\mathcal{Z}$ in $X_1$ such that $\mathcal{Z}_\varepsilon\subset \mathcal{Z}$ for all $\varepsilon\in (0,1)$, see Section~\ref{sec4}. Consequently, there are $\varphi\in\mathcal{Z}$ and a subsequence $(\varphi_{\varepsilon_k})_{k\ge 1}$ of $(\varphi_\varepsilon)_{\varepsilon\in (0,1)}$ such that $\varphi_{\varepsilon_k}\rightharpoonup \varphi$ in $X_1$. We finally combine this convergence with the properties of $\mathcal{Z}$ and $(\varphi_{\varepsilon_k})_{k\ge 1}$ to prove that $\varphi$ is a stationary weak solution to \eqref{a1} as described in Theorem~\ref{thm1} (Section~\ref{sec4}).

\medskip

Theorem~\ref{thm1} only provides the finiteness of the moments of $\varphi$ of order larger than $\lambda$ and thus does not provide much information on its behaviour for small sizes. In fact, the small size behaviour described in Theorem~\ref{thm1}~\textbf{(s2)} does not seem to be accurate. Indeed, formal asymptotics indicate that, if $\varphi$ is a stationary weak solution to \eqref{a1} satisfying the properties~\textbf{(s1)}-\textbf{(s3)} stated in Theorem~\ref{thm1} and 
\begin{subequations}\label{tt}
\begin{equation}
\varphi(x) \sim A x^{-\tau} \;\text{ as }\; x\to 0 \label{t0}
\end{equation}
for some $A>0$ and $\tau>0$, then $\tau$ can be identified and depends on the values of $\alpha$, $\beta$, $\gamma$, and possibly on $B$. Specifically,
\begin{itemize}[label=$-$]
	\item if $\gamma>\alpha$, then
	\begin{equation}
	\tau = \alpha + 1 + m_\star\ , \label{t1}
	\end{equation}
	where $m_\star$ is defined in \eqref{a17} below;
	\item if $\gamma=\alpha<\beta$, then 
	\begin{equation}
	\tau = \alpha+1\ ; \label{t2}
	\end{equation}
	\item if $\gamma=\alpha=\beta$ and $B=B_{1,\nu}$, see \eqref{a15d}, then
	\begin{equation}
	\tau = \alpha + \frac{2}{\nu+3} < \alpha+1\ ; \label{t3}
	\end{equation}
	\item if $\gamma<\alpha$, then
	\begin{equation}
	\tau = \lambda+1-\gamma\ . \label{t4}
	\end{equation}
\end{itemize}
\end{subequations}
In particular, the prediction \eqref{t3} perfectly agrees with \eqref{a12} when $\gamma=\alpha=\beta=\lambda/2\in [0,1/2)$ and $\nu=0$ ($B=B_{1,0}$). On the one hand, \eqref{tt} implies that $\varphi$ may have a non-integrable singularity as $x\to 0$ and, in particular, it is not expected to belong to $X_\alpha$ when $\gamma<\alpha$. On the other hand, different behaviours are predicted in \eqref{tt}, which vary according to the sign of $\gamma-\alpha$, and seem to be sensitive to the behaviour of $B(z)$ as $z\to 0$ when $\gamma=\alpha=\beta$. We shall not attempt a complete proof of \eqref{tt} herein but, as a first step in that direction, we provide additional integrability properties of $\varphi$ which complies with \eqref{tt}.

\begin{proposition}\label{prop2}
Consider $\varrho>0$ and let $\varphi$ be a stationary weak solution to \eqref{a1} satisfying the properties~\textbf{(s1)}-\textbf{(s3)} stated in Theorem~\ref{thm1}.
\begin{itemize}
\item[\textbf{(m1)}] If $\gamma>\alpha$, then $\varphi\in X_m$ for any $m>\alpha+ m_\star$, where
\begin{equation}
m_\star := \inf\left\{ m\in \mathbb{R}\ : \ B\in L^1((0,1),z^m\mathrm{d}z) \right\} \le \frac{1-p_0}{p_0} < 0\ . \label{a17}
\end{equation}
Moreover, if $m_\star>-\infty$ and $B\notin L^1((0,1),z^{m_\star}\mathrm{d}z)$, then $\varphi\notin X_{\alpha+m_\star}$;
\item[\textbf{(m2)}] if $\gamma=\alpha<\beta$, then $\varphi\in X_m$ for any $m\ge \beta$;
\item[\textbf{(m3)}] if $\gamma=\alpha=\beta$, then $\varphi\in X_m$ for any $m\ge \alpha$;
\item[\textbf{(m4)}] if $\gamma<\alpha$, then $\varphi\in X_m$ for any $m>\lambda-\gamma$.
\end{itemize}
\end{proposition}

The proof of Proposition~\ref{prop2} is carried out in Section~\ref{sec5} and relies on the choice of suitable test functions in Theorem~\ref{thm1}~\textbf{(s3)}. Comparing \eqref{tt} and Proposition~\ref{prop2} reveals that the properties~\textbf{(m2)} and~\textbf{(m3)} are not optimal. Improving Proposition~\ref{prop2} so that it matches \eqref{tt} in these cases seems to require a finer analysis which we have yet been unable to set up. We however hope to return to this problem in the near future.

\section{A truncated approximation}\label{sec2}
%
\newcounter{NumConstB}

Let $\varrho>0$ and assume that $K$, $a$, and $b$ are coagulation and fragmentation coefficients satisfying \eqref{a14} and \eqref{a15}. Also, let $f^{in}$ be an initial condition satisfying
\begin{equation}
f^{in}\in X_0^+ \cap X_{2+\gamma} \;\;\text{  with }\;\; M_1(f^{in}) = \varrho\ . \label{b0}
\end{equation}
We now introduce the approximation to \eqref{a1} we are going to work with in this section. Besides requiring a positive lower bound on the coagulation kernel and the overall fragmentation rate as already mentioned, we also truncate both of them as in \cite{EMRR05}. Specifically, we fix a positive integer $j\ge 2$ and a positive real number $\varepsilon\in (0,1)$ and set
\begin{align}
K_{j,\varepsilon}(x,y) & := 2\varepsilon K_0 + K\left( \min\{x,j\} , \min\{y,j\} \right)\ , \qquad (x,y)\in (0,\infty)^2\ , \label{b1} \\
a_{j,\varepsilon}(x) & := a_0 \left( \min\{x,j\}^\gamma + \varepsilon^2 \right)\ , \qquad x\in (0,\infty)\ . \label{b2}
\end{align}
Since $K_{j,\varepsilon}$ and $a_{j,\varepsilon}$ are bounded, we may proceed as in \cite{BLLxx, EMRR05, Stew89, Walk02} to show, by a Banach fixed point argument in $X_0=L^1(0,\infty)$, that there is a unique non-negative strong solution 
\begin{equation*}
f_{j,\varepsilon}\in C^1([0,\infty);X_0)
\end{equation*}
to the coagulation-fragmentation equation 
\begin{subequations}\label{b70}
\begin{align}
\partial_t f_{j,\varepsilon} & = \mathcal{C}_{j,\varepsilon}f_{j,\varepsilon} + \mathcal{F}_{j,\varepsilon}f_{j,\varepsilon}\ , \qquad (t,x) \in (0,\infty)^2\ , \label{b70a} \\
f_{j,\varepsilon}(0) & = f^{in}\ , \qquad x\in (0,\infty)\ , \label{b70b}
\end{align}
\end{subequations}
where the coagulation and fragmentation operators $\mathcal{C}_{j,\varepsilon}$ and $\mathcal{F}_{j,\varepsilon}$ are given by \eqref{a1b} with $K_{j,\varepsilon}$ instead of $K$ and \eqref{a1c} with $a_{j,\varepsilon}$ instead of $a$, respectively. A first consequence of \eqref{b70a} is that, for $t\ge 0$ and $\vartheta\in L^\infty(0,\infty)$, 
\begin{equation}
\begin{split}
\frac{d}{dt} \int_0^\infty \vartheta(x) f_{j,\varepsilon}(t,x)\ \mathrm{d}x & = \frac{1}{2} \int_0^\infty \int_0^\infty K_{j,\varepsilon}(x,y) \chi_\vartheta(x,y) f_{j,\varepsilon}(t,x) f_{j,\varepsilon}(t,y)\ \mathrm{d}y \mathrm{d}x \\
& \qquad - \int_0^\infty a_{j,\varepsilon}(y) N_\vartheta(y) f_{j,\varepsilon}(t,y)\ \mathrm{d}y \ ,
\end{split} \label{b5}
\end{equation}
where
\begin{subequations}\label{b6}
\begin{align}
\chi_\vartheta(x,y) & := \vartheta(x+y) - \vartheta(x) - \vartheta(y)\ , \qquad (x,y)\in (0,\infty)^2\ , \label{b6a} \\
N_\vartheta(y) & := \vartheta(y) - \int_0^y \vartheta(x) b(x,y)\ \mathrm{d}x\ , \qquad y\in (0,\infty)\ , \label{b6b} 
\end{align}
Owing to \eqref{a15b}, an alternative formula for $N_\vartheta$ reads
\begin{equation}
N_\vartheta(y) = \vartheta(y) - \int_0^1 \vartheta(yz) B(z)\ \mathrm{d}z\ , \qquad y\in (0,\infty)\ . \label{b6c}
\end{equation}
\end{subequations}
For the particular choice $\vartheta(x)=\vartheta_m(x):=x^m$, $x>0$, for some $m\in\mathbb{R}$, we set $\chi_m:=\chi_{\vartheta_m}$ and $N_m:=N_{\vartheta_m}$ for simplicity.

Owing to the boundedness of $K_{j,\varepsilon}$ and $a_{j,\varepsilon}$ and the integrability \eqref{a15c} of $B$ over $(0,1)$, we infer from \eqref{b5} by an approximation argument that $f_{j,\varepsilon}$ is mass-conserving; that is, $f_{j,\varepsilon}\in L^\infty((0,\infty),X_1)$ and
\begin{equation}
M_1(f_{j,\varepsilon}(t)) = \varrho\ , \qquad t\ge 0\ . \label{b4}
\end{equation}
Moreover, a similar approximation argument allows us to show that, if $f^{in}\in X_m$ for some $m>1$, then $f_{j,\varepsilon}\in L^\infty((0,T),X_m)$ for any $T>0$. We shall refine this result in the next section.

\medskip

We now derive several estimates for the family $\{f_{j,\varepsilon}\ :\ j\ge 2\ , \ \varepsilon\in (0,1)\}$, which do not depend on $j\ge 2$. We also pay special attention to the dependence on $\varepsilon\in (0,1)$, if any. Throughout this section, $C$ and $C_i$, $i\ge 1$, denote positive constants which depend only on $K_0$, $\alpha$, $\beta$, $a_0$, $\gamma$, $B$, and $\varrho$. Dependence upon additional parameters will be indicated explicitly. For further use, we set
\begin{subequations}\label{b95}
\begin{equation}
\mathfrak{b}_m := \int_0^1 z^m B(z)\ \mathrm{d}z \;\text{ for }\; m>m_\star \;\;\text{ and }\;\; \mathcal{B}_p^p := \int_0^1 B(z)^{p}\ \mathrm{d}z \;\text{ for}\; p\in [1,p_0]\ , \label{b95a}
\end{equation}	
which are finite by \eqref{a15c} and \eqref{a17}, and satisfy 
\begin{equation}
\mathfrak{b}_m < 1 \iff m>1 \label{b95b}
\end{equation}
\end{subequations}
due to \eqref{a15c}. Also, Young's inequality and \eqref{a14} entail that
\begin{equation}
K(x,y) \le K_0 \left( x^\lambda+y^\lambda \right)\ , \qquad (x,y)\in (0,\infty)^2\ . \label{b950}
\end{equation}

\subsection{Moment Estimates}\label{sec2.1}

For $m\in\mathbb{R}$ we set 
\begin{equation}
\mathcal{M}_{m,j,\varepsilon} := \sup_{t\ge 0} \left\{ M_m(f_{j,\varepsilon}(t)) \right\} \in [0,\infty] \label{spirou}
\end{equation}
and begin with the behaviour of $f_{j,\varepsilon}$ for large sizes.

\begin{lemma}\label{lemb1}
Let $m\ge 2$ and assume that $f^{in}\in X_m$. There is a positive constant $\mu_m \ge \Gamma(m+1) \varrho^m$ depending only on $K_0$, $\alpha$, $\beta$, $a_0$, $\gamma$, $B$, $\varrho$, and $m$ such that
\begin{equation*}
\mathcal{M}_{m,j,\varepsilon} \le \max\{ M_m(f^{in}) , \mu_m \}\ .
\end{equation*}
\end{lemma}

\begin{proof}
We first recall that there is $c_m>0$ depending only on $m$ such that
\begin{equation}
\chi_m(x,y) \le c_m \left( x y^{m-1} + x^{m-1} y \right)\ , \qquad (x,y)\in (0,\infty)^2\ , \label{spip}
\end{equation}
see \cite[Lemma~2.3~(ii)]{Carr92} for instance. Let $t>0$. We infer from \eqref{b5} with $\vartheta=\vartheta_m$, \eqref{b95}, \eqref{spip}, and the symmetry of $K$ that
\begin{align*}
\frac{d}{dt} M_m(f_{j,\varepsilon}(t)) & \le \frac{c_m}{2} \int_0^\infty \int_0^\infty K_{j,\varepsilon}(x,y) \left( x^{m-1} y + x y^{m-1} \right) f_{j,\varepsilon}(t,x) f_{j,\varepsilon}(t,y)\ \mathrm{d}y\mathrm{d}x \\
& \qquad - (1 - \mathfrak{b}_m) \int_0^\infty x^m a_{j,\varepsilon}(x) f_{j,\varepsilon}(t,x)\ \mathrm{d}x \\
& \le c_m \int_0^\infty \int_0^\infty x y^{m-1} K_{j,\varepsilon}(x,y) f_{j,\varepsilon}(t,x) f_{j,\varepsilon}(t,y)\ \mathrm{d}y\mathrm{d}x \\ 
& \qquad - a_0 (1 - \mathfrak{b}_m) \int_0^\infty x^m \min\{x,j\}^\gamma f_{j,\varepsilon}(t,x)\ \mathrm{d}x \ .
\end{align*}
On the one hand, by \eqref{b4},
\begin{align*}
\int_0^\infty x^m \min\{x,j\}^\gamma f_{j,\varepsilon}(t,x)\ \mathrm{d}x & \ge \int_1^\infty x^m \min\{x,j\}^\gamma f_{j,\varepsilon}(t,x)\ \mathrm{d}x \\
& \ge \int_1^\infty x^m f_{j,\varepsilon}(t,x)\ \mathrm{d}x \\
& = M_m(f_{j,\varepsilon}(t)) - \int_0^1 x^m f_{j,\varepsilon}(t,x)\ \mathrm{d}x \\
& \ge M_m(f_{j,\varepsilon}(t)) - \varrho\ .
\end{align*}
On the other hand, it follows from \eqref{b4} and H\"older's and Young's inequalities that
\begin{align*}
\int_0^\infty \int_0^\infty x y^{m-1} f_{j,\varepsilon}(t,x) f_{j,\varepsilon}(t,y)\ \mathrm{d}y\mathrm{d}x & \le \varrho M_{m-1}(f_{j,\varepsilon}(t)) \\& \le \varrho M_{m}(f_{j,\varepsilon}(t))^{(m-2)/(m-1)} M_{1}(f_{j,\varepsilon}(t))^{1/(m-1)} \\
& \le \frac{a_0(1-\mathfrak{b}_m)}{8 c_m K_0} M_{m}(f_{j,\varepsilon}(t)) + C(m)\ .
\end{align*}
Similarly, 
\begin{align*}
& \int_0^\infty \int_0^\infty x y^{m-1} \min\{y,j\}^\lambda f_{j,\varepsilon}(t,x) f_{j,\varepsilon}(t,y)\ \mathrm{d}y\mathrm{d}x \\
& \qquad \le \varrho M_{m+\lambda-1}(f_{j,\varepsilon}(t)) \\
& \qquad \le \varrho M_{m}(f_{j,\varepsilon}(t))^{(m+\lambda-2)/(m-1)} M_{1}(f_{j,\varepsilon}(t))^{(1-\lambda)/(m-1)} \\
& \qquad \le \frac{a_0(1-\mathfrak{b}_m)}{4 c_m K_0} M_{m}(f_{j,\varepsilon}(t)) + C(m)\ ,
\end{align*}
and 
\begin{align*}
&\int_0^\infty \int_0^\infty x y^{m-1} \min\{x,j\}^\lambda f_{j,\varepsilon}(t,x) f_{j,\varepsilon}(t,y)\ \mathrm{d}y\mathrm{d}x \\
& \qquad \le M_{1+\lambda}(f_{j,\varepsilon}(t)) M_{m-1}(f_{j,\varepsilon}(t)) \\
& \qquad \le \varrho M_{m}(f_{j,\varepsilon}(t))^{(m+\lambda-2)/(m-1)} M_{1}(f_{j,\varepsilon}(t))^{(m-\lambda)/(m-1)} \\
& \qquad \le \frac{a_0(1-\mathfrak{b}_m)}{4 c_m K_0} M_{m}(f_{j,\varepsilon}(t)) + C(m)\ .
\end{align*}
Collecting the previous inequalities and using \eqref{b950}, we obtain
\begin{align*}
\frac{d}{dt} M_m(f_{j,\varepsilon}(t)) & \le \frac{2+\varepsilon}{4} a_0 (1-\mathfrak{b}_m) M_m(f_{j,\varepsilon}(t)) +C(m) \\
& \qquad - a_0 (1-\mathfrak{b}_m) \left( M_m(f_{j,\varepsilon}(t)) - \varrho \right) \\
& \le - \frac{a_0}{4} (1-\mathfrak{b}_m) M_m(f_{j,\varepsilon}(t)) + C(m)\ .
\end{align*}
Integrating the previous differential inequality gives
\begin{equation*}
M_m(f_{j,\varepsilon}(t)) \le e^{-a_0(1-\mathfrak{b}_m)t /4} M_m(f^{in}) + \frac{4 C(m)}{a_0(1-\mathfrak{b}_m)} \left( 1 - e^{-a_0(1-\mathfrak{b}_m)t /4} \right)
\end{equation*}
for $t\ge 0$. Therefore,
\begin{equation*}
M_m(f_{j,\varepsilon}(t)) \le \max\left\{ M_m(f^{in}) , \frac{4 C(m)}{a_0(1-\mathfrak{b}_m)} \right\}\ , \qquad t\ge 0\ ,
\end{equation*}
from which Lemma~\ref{lemb1} follows. 
\end{proof}	

From now on, we fix a positive real number 
\begin{subequations}\label{b98}
\begin{equation}
\sigma>\max\left\{ 1 , \varrho , \mu_2 , \mu_{2+\gamma} \right\} \label{b98a}
\end{equation} 
such that
\begin{equation}
\max\left\{ M_2(f^{in}) , M_{2+\gamma}(f^{in}) \right\} \le \sigma\ . \label{b98b}
\end{equation}
\end{subequations}
A first consequence of \eqref{b4}, \eqref{b98}, Lemma~\ref{lemb1}, and H\"older's inequality is that
\begin{equation}
\mathcal{M}_{1+\gamma,j,\varepsilon} \le \sigma \;\text{ and }\; \mathcal{M}_{2,j,\varepsilon} \le \sigma\ . \label{b97}
\end{equation}

Next, owing to \eqref{b4}, \eqref{b98}, and \eqref{b97}, another application of H\"older's inequality provides a similar bound for moments of order $m\in (1,2)$, which we report now.

\begin{corollary}\label{corb1b}
For $m\in (1,2)$,
\begin{equation*}
\mathcal{M}_{m,j,\varepsilon} \le \sigma\ .
\end{equation*}
\end{corollary}

We next turn to the behaviour for small sizes and, to this end, derive estimates for moments of order smaller than one.

\begin{lemma}\label{lemb2}
Let $m\in (\lambda,1)$. There is $\mu_m\ge \Gamma(m+1) \varrho^m$ depending only $K_0$, $\alpha$, $\beta$, $a_0$, $\gamma$, $B$, $\varrho$, and $m$ such that
\begin{equation*}
\mathcal{M}_{m,j,\varepsilon} \le \max\{ M_m(f^{in}) , \mu_m + \sigma \}\ .
\end{equation*}
\end{lemma}

\begin{proof}
Let $m\in (\lambda,1)$ and $t>0$. We first argue as in \cite[Lemma~3.1]{FoLa05} to estimate the contribution of the coagulation term to the time evolution of $M_m(f_{j,\varepsilon})$, see also \cite[Lemma~8.2.12]{BLLxx}. More precisely, since $j\ge 2$, $\chi_m\le 0$, and $K_{j,\varepsilon}(x,y) \ge 2K_0 (xy)^{\lambda/2}$ for $(x,y)\in (0,1)^2$, we obtain
\begin{align}
& - \frac{1}{2K_0} \int_0^\infty \int_0^\infty K_{j,\varepsilon}(x,y) \chi_m(x,y) f_{j,\varepsilon}(t,x) f_{j,\varepsilon}(t,y)\ \mathrm{d}y\mathrm{d}x \nonumber \\
& \qquad \ge \frac{1}{2K_0} \int_0^1 \int_0^1 \left[ x^m + y^m - (x+y)^m \right] K_{j,\varepsilon}(x,y) f_{j,\varepsilon}(t,x) f_{j,\varepsilon}(t,y)\ \mathrm{d}y\mathrm{d}x \nonumber \\
& \qquad \ge P_{j,\varepsilon}(t) := \int_0^1 \int_0^1 \left[ x^m + y^m - (x+y)^m \right] (xy)^{\lambda/2} f_{j,\varepsilon}(t,x) f_{j,\varepsilon}(t,y)\ \mathrm{d}y\mathrm{d}x \ . \label{b9}
\end{align}
Since $m<1$, it follows from the convexity of $x\mapsto x^{m-1}$ that, for $(x,y)\in (0,\infty)^2$, 
\begin{align*}
x^m + y^m - (x+y)^m & = x \left[ x^{m-1} - (x+y)^{m-1} \right] + y \left[ y^{m-1} - (x+y)^{m-1} \right] \\ 
& \ge 2 (1-m) x y (x+y)^{m-2}\ .
\end{align*}
Therefore, 
\begin{equation*}
P_{j,\varepsilon}(t) \ge 2(1-m) \int_0^1 \int_0^1 (x+y)^{m-2} (xy)^{(2+\lambda)/2} f_{j,\varepsilon}(t,x) f_{j,\varepsilon}(t,y)\ \mathrm{d}y\mathrm{d}x \ .
\end{equation*}
Introducing
$$
x_i := i^{-2/(m-\lambda)} \;\text{ and }\; P_{j,\varepsilon}(t,i) := \int_{x_{i+1}}^{x_i} x^{(2+\lambda)/2} f_{j,\varepsilon}(t,x)\ \mathrm{d}x\ , \qquad i\ge 1\ ,
$$
we further obtain
\begin{align}
P_{j,\varepsilon}(t) & \ge 2(1-m) \sum_{i=1}^\infty \int_{x_{i+1}}^{x_i} \int_{x_{i+1}}^{x_i} (x+y)^{m-2} (xy)^{(2+\lambda)/2} f_{j,\varepsilon}(t,x) f_{j,\varepsilon}(t,y)\ \mathrm{d}y\mathrm{d}x \nonumber \\
& \ge 2^{m-1} (1-m) \sum_{i=1}^\infty x_i^{m-2} P_{j,\varepsilon}(t,i)^2\ . \label{b10}
\end{align}
It next follows from the Cauchy-Schwarz inequality that
\begin{align}
\int_0^1 x^m f_{j,\varepsilon}(t,x)\ \mathrm{d}x & = \sum_{i=1}^\infty \int_{x_{i+1}}^{x_i} x^m f_{j,\varepsilon}(t,x)\ \mathrm{d}x \le \sum_{i=1}^\infty x_{i+1}^{(2m-2-\lambda)/2} P_{j,\varepsilon}(t,i) \nonumber \\
& \le \left( \sum_{i=1}^\infty x_{i+1}^{2m-2-\lambda} x_i^{2-m} \right)^{1/2} \left( \sum_{i=1}^\infty x_i^{m-2} P_{j,\varepsilon}(t,i)^2 \right)^{1/2}\ . \label{b11}
\end{align}
Since 
\begin{equation*}
x_{i+1}^{2m-2-\lambda} x_i^{2-m} \le (2i)^{2(2+\lambda-2m)/(m-\lambda)} i^{-2(2-m)/(m-\lambda)} = 4^{(2+\lambda-2m)/(m-\lambda)} i^{-2}\ ,
\end{equation*}
the series in the right-hand side of \eqref{b11} converges and we deduce from \eqref{b10} and \eqref{b11} that
\refstepcounter{NumConstB}\label{cstB1}
\begin{equation}
P_{j,\varepsilon}(t) \ge C_{\ref{cstB1}}(m) \left( \int_0^1 x^m f_{j,\varepsilon}(t,x)\ \mathrm{d}x \right)^2\ . \label{b12}
\end{equation}
Furthermore, as
\begin{align*}
M_m(f_{j,\varepsilon}(t)) & = \int_0^1 x^m f_{j,\varepsilon}(t,x)\ \mathrm{d}x + \int_1^\infty x^m f_{j,\varepsilon}(t,x)\ \mathrm{d}x \\
& \le \int_0^1 x^m f_{j,\varepsilon}(t,x)\ \mathrm{d}x + \int_1^\infty x f_{j,\varepsilon}(t,x)\ \mathrm{d}x \\
& \le \int_0^1 x^m f_{j,\varepsilon}(t,x)\ \mathrm{d}x + \varrho
\end{align*}
by \eqref{b4}, we infer from Young's inequality that
\begin{equation}
\left( \int_0^1 x^m f_{j,\varepsilon}(t,x)\ \mathrm{d}x \right)^2 \ge \frac{M_m(f_{j,\varepsilon}(t))^2}{2} - \varrho^2\ . \label{b13}
\end{equation}
Combining \eqref{b9}, \eqref{b12}, and \eqref{b13} provides the existence of two positive constants \refstepcounter{NumConstB}\label{cstB2} $C_{\ref{cstB2}}(m)$ and \refstepcounter{NumConstB}\label{cstB3}
$C_{\ref{cstB3}}(m)$ such that
\begin{equation}
\frac{1}{2} \int_0^\infty \int_0^\infty K_{j,\varepsilon}(x,y) \chi_m(x,y) f_{j,\varepsilon}(t,x) f_{j,\varepsilon}(t,y)\ \mathrm{d}y\mathrm{d}x \le C_{\ref{cstB2}}(m) - C_{\ref{cstB3}}(m) M_m(f_{j,\varepsilon}(t))^2\ . \label{b13b}
\end{equation}
Consequently, recalling that $\mathfrak{b}_m>1$ by \eqref{b95} as $m<1$, it follows from \eqref{b5} with $\vartheta=\vartheta_m$, \eqref{b97}, \eqref{b13b}, and Young's inequality that \refstepcounter{NumConstB}\label{cstB4}
\begin{align*}
\frac{d}{dt} M_m(f_{j,\varepsilon}(t)) & = \frac{1}{2} \int_0^\infty \int_0^\infty K_{j,\varepsilon}(x,y)\chi_m(x,y) f_{j,\varepsilon}(t,x) f_{j,\varepsilon}(t,y)\ \mathrm{d}y\mathrm{d}x \\
& \qquad + a_0 (\mathfrak{b}_m-1) \int_0^\infty x^m \left( \varepsilon^2 + \min\{x,j\}^\gamma \right) f_{j,\varepsilon}(t,x)\ \mathrm{d}x \\
& \le C_{\ref{cstB2}}(m) - C_{\ref{cstB3}}(m) M_m(f_{j,\varepsilon}(t))^2 + a_0 \mathfrak{b}_m M_{m+\gamma}(f_{j,\varepsilon}(t)) \\
& \qquad + a_0 \mathfrak{b}_m \varepsilon^2 M_m(f_{j,\varepsilon}(t)) \\
& \le C_{\ref{cstB2}}(m) - C_{\ref{cstB3}}(m) M_m(f_{j,\varepsilon}(t))^2 + \frac{a_0 \mathfrak{b}_m \gamma}{\gamma+1-m} M_{\gamma+1}(f_{j,\varepsilon}(t)) \\
& \qquad + a_0 \mathfrak{b}_m \left[ \frac{1-m}{\gamma+1-m} + 1 \right] M_m(f_{j,\varepsilon}(t)) \\
& \le C_{\ref{cstB4}}(m) \left[ 1 + M_m(f_{j,\varepsilon}(t)) + \sigma \right] - C_{\ref{cstB3}}(m) M_m(f_{j,\varepsilon}(t))^2\ .
\end{align*}
As
\begin{equation*}
M_m(f_{j,\varepsilon}(t)) \le \frac{C_{\ref{cstB3}}(m)}{2C_{\ref{cstB4}}(m)} M_m(f_{j,\varepsilon}(t))^2 + \frac{C_{\ref{cstB4}}(m)}{2C_{\ref{cstB3}}(m)}\ ,
\end{equation*}
we finally obtain
\begin{equation*}
\frac{d}{dt} M_m(f_{j,\varepsilon}(t)) \le -C_{\ref{cstB4}}(m) M_m(f_{j,\varepsilon}(t)) + C_{\ref{cstB4}}(m) \left[ 1 + \frac{C_{\ref{cstB4}}(m)}{C_{\ref{cstB3}}(m)} +  \sigma \right]\ , \qquad t\ge 0\ .
\end{equation*}
Integrating the previous differential inequality gives
\begin{equation*}
M_m(f_{j,\varepsilon}(t)) \le e^{-C_{\ref{cstB4}}(m) t} M_m(f^{in}) + \left[ 1 + \frac{C_{\ref{cstB4}}(m)}{C_{\ref{cstB3}}(m)} +  \sigma \right] \left( 1 - e^{-C_{\ref{cstB4}}(m) t} \right)\ , \qquad t\ge 0\ .
\end{equation*}
Therefore, 
\begin{equation*}
M_m(f_{j,\varepsilon}(t)) \le \max\left\{ M_m(f^{in}) , 1 + \frac{C_{\ref{cstB4}}(m)}{C_{\ref{cstB3}}(m)} + \sigma \right\}\ , \qquad t\ge 0\ , 
\end{equation*}
from which Lemma~\ref{lemb2} follows.
\end{proof}

The next step is devoted to the derivation of additional estimates for small sizes but now with a strong dependence on $\varepsilon$.

\begin{lemma}\label{lemb3}
There is $\mu_0\ge 1$ depending only on $K_0$, $a_0$, $B$, and $\varrho$ such that
\begin{equation*}
\mathcal{M}_{0,j,\varepsilon} \le \max\left\{ M_0(f^{in}) , \sigma + \frac{\mu_0}{\varepsilon} \right\}\ .
\end{equation*}
\end{lemma}

\begin{proof}
It follows from \eqref{b5} with $\vartheta \equiv 1$, \eqref{b95}, \eqref{b97}, and Young's inequality that, for $t\ge 0$,
\begin{align*}
\frac{d}{dt} M_0(f_{j,\varepsilon}(t)) & = - \frac{1}{2} \int_0^\infty \int_0^\infty K_{j,\varepsilon}(x,y) f_{j,\varepsilon}(t,x) f_{j,\varepsilon}(t,y)\ \mathrm{d}y\mathrm{d}x \\
& \qquad - a_0 (1-\mathfrak{b}_0) \int_0^\infty \left( \varepsilon^2 + \min\{x,j\}^\gamma \right) f_{j,\varepsilon}(t,x)\ \mathrm{d}x \\
& \le - \varepsilon K_0 M_0(f_{j,\varepsilon}(t))^2 + a_0 \mathfrak{b}_0 \varepsilon^2 M_0(f_{j,\varepsilon}(t)) + a_0 \mathfrak{b}_0 M_\gamma(f_{j,\varepsilon}(t)) \\
& \le - \varepsilon K_0 M_0(f_{j,\varepsilon}(t))^2 + \frac{\gamma a_0 \mathfrak{b}_0}{1+\gamma} M_{\gamma+1}(f_{j,\varepsilon}(t)) \\
& \qquad + a_0 \mathfrak{b}_0 \left( \frac{1}{1+\gamma} + 1 \right) M_0(f_{j,\varepsilon}(t)) \\
& \le - \varepsilon K_0 M_0(f_{j,\varepsilon}(t))^2 + a_0 \mathfrak{b}_0 \sigma + 2 a_0 \mathfrak{b}_0 M_0(f_{j,\varepsilon}(t))\ .
\end{align*}
By the Cauchy-Schwarz inequality,
\begin{equation*}
4 a_0 \mathfrak{b}_0 M_0(f_{j,\varepsilon}(t)) \le \varepsilon K_0 M_0(f_{j,\varepsilon}(t))^2 + \frac{4 a_0^2 \mathfrak{b}_0^2}{\varepsilon K_0}\ .
\end{equation*}
Hence
\begin{equation*}
\frac{d}{dt} M_0(f_{j,\varepsilon}(t)) + 2 a_0 \mathfrak{b}_0 M_0(f_{j,\varepsilon}(t)) \le a_0 \mathfrak{b}_0 \sigma + \frac{4 a_0^2 \mathfrak{b}_0^2}{\varepsilon K_0}\ , \qquad t\ge 0\ .
\end{equation*}
Integrating this differential inequality, we find
\begin{align*}
M_0(f_{j,\varepsilon}(t)) & \le M_0(f^{in}) e^{-2 a_0 \mathfrak{b}_0 t} + \left( \frac{\sigma}{2} + \frac{2 a_0 \mathfrak{b}_0}{\varepsilon K_0} \right) \left( 1 -  e^{-2 a_0 \mathfrak{b}_0 t} \right) \\ 
& \le \max\left\{ M_0(f^{in}) , \sigma + \frac{2 a_0 \mathfrak{b}_0}{\varepsilon K_0} \right\}
\end{align*}
for $t\ge 0$, as claimed.
\end{proof}

The previous result actually extends to some moments of negative order.

\begin{lemma}\label{lemb30}
Let $m\in (m_\star,0)$ and set
\begin{equation}
\varepsilon_{m,\sigma} := \frac{1}{\sigma} \min\left\{ 1 , \frac{K_0 \varrho^2}{4 a_0 \mathfrak{b}_m} \right\}\ , \label{champignac}
\end{equation}
where $m_\star$ and $\sigma$ are defined in \eqref{a17} and \eqref{b98}, respectively. There is $\mu_m>0$ depending only on $K_0$, $a_0$, $\varrho$, $B$, and $m$ such that, if $f^{in}\in X_m$ and $\varepsilon\in (0,\varepsilon_{m,\sigma})$, then
\begin{equation*}
\mathcal{M}_{m,j,\varepsilon} \le \max\left\{ M_m(f^{in}) , \mu_m \sigma^2 \varepsilon^{-(\gamma+2-2m)/\gamma} \right\}\ .
\end{equation*}
We may also assume that $\mu_m\ge \Gamma(m+1) \varrho^m$ when $m>-1$.
\end{lemma}

\begin{proof}
For $\delta\in (0,1)$, we set $\vartheta_{m,\delta}(x) := (x+\delta)^m$, $x>0$, and notice that 
\begin{equation*}
\chi_{\vartheta_{m,\delta}}(x,y) \le - (x+\delta)^m \le 0\ , \qquad (x,y)\in (0,\infty)^2\ .
\end{equation*}	
Let $\varepsilon\in (0,\varepsilon_{m,\sigma})$ and $t>0$. We infer from \eqref{b5} with $\vartheta=\vartheta_{m,\delta}$ that 
\begin{align*}
\frac{d}{dt} \int_0^\infty \vartheta_{m,\delta}(x) f_{j,\varepsilon}(t,x)\ \mathrm{d}x & \le - \varepsilon K_0 \int_0^\infty \int_0^\infty (x+\delta)^m f_{j,\varepsilon}(t,x) f_{j,\varepsilon}(t,y)\ \mathrm{d}y\mathrm{d}x \\
& \qquad + \int_0^\infty a_{j,\varepsilon}(y) f_{j,\varepsilon}(t,y) \int_0^y \vartheta_{m,\delta}(x) b(x,y)\ \mathrm{d}x\mathrm{d}y\ . 
\end{align*}
On the one hand, by \eqref{b4}, \eqref{b97}, and the Cauchy-Schwarz inequality,
\begin{equation*}
\varrho^2 = M_1(f_{j,\varepsilon}(t))^2 \le M_0(f_{j,\varepsilon}(t)) M_2(f_{j,\varepsilon}(t)) \le \sigma M_0(f_{j,\varepsilon}(t))\ ,
\end{equation*}
so that
\begin{align*}
U_{j,\varepsilon}(t) & := \int_0^\infty \int_0^\infty (x+\delta)^m f_{j,\varepsilon}(t,x) f_{j,\varepsilon}(t,y)\ \mathrm{d}y\mathrm{d}x \\
& = M_0(f_{j,\varepsilon}(t)) \int_0^\infty (x+\delta)^m f_{j,\varepsilon}(t,x)\ \mathrm{d}x \\
& \ge \frac{\varrho^2}{\sigma} \int_0^\infty \vartheta_{m,\delta}(x) f_{j,\varepsilon}(t,x)\ \mathrm{d}x \ .
\end{align*}
On the other hand, we infer from \eqref{a15b}, \eqref{b97}, and the negativity of $m$ that
\begin{align*}
V_{j,\varepsilon}(t) & := \int_0^\infty a_{j,\varepsilon}(y) f_{j,\varepsilon}(t,y) \int_0^y \vartheta_{m,\delta}(x) b(x,y)\ \mathrm{d}x\mathrm{d}y \\
& = \int_0^\infty a_{j,\varepsilon}(y) f_{j,\varepsilon}(t,y) \int_0^1 (yz+\delta)^m B(z)\ \mathrm{d}z\mathrm{d}y \\
& \le \int_0^\infty a_{j,\varepsilon}(y) f_{j,\varepsilon}(t,y) \int_0^1 (yz+\delta z)^m B(z)\ \mathrm{d}z\mathrm{d}y \\
& \le a_0 \mathfrak{b}_m \int_0^\infty (x+\delta)^m \left( \varepsilon^2 + \min\{x,j\}^\gamma \right) f_{j,\varepsilon}(t,x)\ \mathrm{d}x \ .
\end{align*}
Since
\begin{align*}
\int_0^\infty (x+\delta)^m \min\{x,j\}^\gamma f_{j,\varepsilon}(t,x)\ \mathrm{d}x & \le \varepsilon^2 \int_0^{\varepsilon^{2/\gamma}} (x+\delta)^m f_{j,\varepsilon}(t,x)\ \mathrm{d}x \\
& \qquad + \varepsilon^{2(m-1)/\gamma} \int_{\varepsilon^{2/\gamma}}^\infty x^{\gamma+1} f_{j,\varepsilon}(t,x)\ \mathrm{d}x \\
& \le \varepsilon^2 \int_0^\infty (x+\delta)^m f_{j,\varepsilon}(t,x)\ \mathrm{d}x + \varepsilon^{2(m-1)/\gamma} \mathcal{M}_{\gamma+1,j,\varepsilon} \\
& \le \varepsilon^2 \int_0^\infty (x+\delta)^m f_{j,\varepsilon}(t,x)\ \mathrm{d}x + \sigma \varepsilon^{2(m-1)/\gamma}
\end{align*}
by \eqref{b97}, we further obtain
\begin{equation*}
V_{j,\varepsilon}(t) \le a_0 \mathfrak{b}_m \left( 2 \varepsilon^2 \int_0^\infty \vartheta_{m,\delta}(x) f_{j,\varepsilon}(t,x)\ \mathrm{d}x + \sigma \varepsilon^{2(m-1)/\gamma} \right)\ . 
\end{equation*}
Collecting the previous estimates and using the definition \eqref{champignac} of $\varepsilon_{m,\sigma}$ lead us to the differential inequality 
\begin{align*}
\frac{d}{dt} \int_0^\infty \vartheta_{m,\delta}(x) f_{j,\varepsilon}(t,x)\ \mathrm{d}x & \le - \frac{\varepsilon K_0 \varrho^2}{\sigma} \int_0^\infty \vartheta_{m,\delta}(x) f_{j,\varepsilon}(t,x)\ \mathrm{d}x \\
& \qquad + 2 a_0 \mathfrak{b}_m \varepsilon^2 \int_0^\infty \vartheta_{m,\delta}(x) f_{j,\varepsilon}(t,x)\ \mathrm{d}x + a_0 \mathfrak{b}_m \sigma \varepsilon^{2(m-1)/\gamma} \\
& \le 2 a_0 \mathfrak{b}_m \varepsilon (\varepsilon-2\varepsilon_{m,\sigma}) \int_0^\infty \vartheta_{m,\delta}(x) f_{j,\varepsilon}(t,x)\ \mathrm{d}x + a_0 \mathfrak{b}_m \sigma \varepsilon^{2(m-1)/\gamma} \\
&  \le - 2 a_0 \mathfrak{b}_m \varepsilon \varepsilon_{m,\sigma} \int_0^\infty \vartheta_{m,\delta}(x) f_{j,\varepsilon}(t,x)\ \mathrm{d}x + a_0 \mathfrak{b}_m \sigma \varepsilon^{2(m-1)/\gamma}\ .
\end{align*}
After integration with respect to time, we end up with
\begin{align*}
\int_0^\infty \vartheta_{m,\delta}(x) f_{j,\varepsilon}(t,x)\ \mathrm{d}x & \le e^{- 2 a_0 \mathfrak{b}_m \varepsilon \varepsilon_{m,\sigma} t} \int_0^\infty \vartheta_{m,\delta}(x) f^{in}(x)\ \mathrm{d}x \\
& \qquad + \frac{\sigma \varepsilon^{-(\gamma+2-2m)/\gamma}}{2 \varepsilon_{m,\sigma}} \left( 1 - e^{- 2 a_0 \mathfrak{b}_m \varepsilon \varepsilon_{m,\sigma} t} \right) \\
& \le \max\left\{ M_m(f^{in}) , \mu_m \sigma^2 \varepsilon^{-(\gamma+2-2m)/\gamma} \right\}\ , \qquad t\ge 0\ .
\end{align*} 
Since the right-hand side of the previous inequality does not depend on $\delta\in (0,1)$ and is finite, we may pass to the limit as $\delta\to 0$ and thereby complete the proof of Lemma~\ref{lemb30}. 
\end{proof}

\begin{remark}\label{rem1}
It is worth mentioning here that the positivity of $\gamma$ is only used in the proof of Lemma~\ref{lemb30}. 
\end{remark}

\subsection{Integrability Estimates}\label{sec2.3}

We now turn to weighted $L^p$-estimates and actually derive two different estimates, one depending on $\varepsilon$ but not on $t$, and the other one depending on $t$ but not on $\varepsilon$. For $m\ge 0$, $p\ge 1$, and $h\in L^p((0,\infty),x^m\mathrm{d}x)$, we set 
\begin{equation}
L_{m,p}(h) := \int_0^\infty x^m |h(x)|^p\ \mathrm{d}x\ . \label{fantasio}
\end{equation}

\begin{lemma}\label{lemb3b}
Consider $m\in (\lambda,1)$ and $p\in (1,p_0]$ satisfying 
\begin{equation}
1<p < \frac{m+1}{\lambda+1} \;\text{ and }\; p \le\frac{m+\gamma}{\gamma}\ , \label{b14}
\end{equation}
and assume that $f^{in}\in L^p((0,\infty),x^m\mathrm{d}x)$. Then
\begin{equation}
L_{m,p}(f_{j,\varepsilon}(t)) \le \max\left\{ L_{m,p}(f^{in}) , \frac{S_{j,\varepsilon}(m,p)}{\varepsilon^2} \right\} \label{b15}
\end{equation}
and
\begin{equation}
\frac{1}{t} \int_0^t  \int_0^\infty x^m \min\{x,j\}^\gamma (f_{j,\varepsilon}(s,x))^p\ \mathrm{d}x\mathrm{d}s \le \frac{1}{a_0 t} L_{m,p}(f^{in}) +  S_{j,\varepsilon}(m,p) \ , \label{b16}
\end{equation}
where
\begin{equation*}
S_{j,\varepsilon}(m,p) := 2^p \mathcal{B}_p^p \left( \mathcal{M}_{(m+1+\gamma-p)/p,j,\varepsilon}^p + \mathcal{M}_{(m+1+\gamma p-p)/p,j,\varepsilon}^p + \varepsilon^2 \mathcal{M}_{(m+1-p)/p,j,\varepsilon}^p \right)
\end{equation*}
and $\mathcal{B}_p$ is defined in \eqref{b95a}.
\end{lemma}

\begin{proof} We first note that \eqref{a15a} and \eqref{b14} ensure that 
\begin{equation*}
\frac{m+1+\gamma p -p}{p} \ge \frac{m+1+\gamma-p}{p} \ge \frac{m+1-p}{p} > \lambda\ ,
\end{equation*}
so that $S_{j,\varepsilon}(m,p)$ is well-defined and finite by Lemma~\ref{lemb2}. 

Let $t>0$. We first deal with the contribution of the coagulation term. As already observed in \cite{BLLxx, Dubo94b, LaMi02c, MiRR03}, the sublinearity of $x\mapsto x^m$ and the monotonicity of $x\mapsto K_{j,\varepsilon}(x,y)$ for all $y>0$ allow us to show that this contribution is negative. Indeed, it follows from the inequality
\begin{equation*}
(x+y)^m \le x^m + y^m \ , \qquad (x,y)\in (0,\infty)^2\ , 
\end{equation*}
the symmetry of $K_{j,\varepsilon}$, and Fubini's theorem that
\begin{align*}
P_{j,\varepsilon}(t) & := p \int_0^\infty x^m (f_{j,\varepsilon}(t,x))^{p-1} (\mathcal{C}_{j,\varepsilon} f_{j,\varepsilon})(t,x)\ \mathrm{d}x \\
& = \frac{p}{2} \int_0^\infty \int_0^\infty (x+y)^m K_{j,\varepsilon}(x,y) (f_{j,\varepsilon}(t,x+y))^{p-1} f_{j,\varepsilon}(t,x) f_{j,\varepsilon}(t,y)\ \mathrm{d}y\mathrm{d}x \\
& \quad - p \int_0^\infty \int_0^\infty x^m K_{j,\varepsilon}(x,y) (f_{j,\varepsilon}(t,x))^p f_{j,\varepsilon}(t,y)\ \mathrm{d}y\mathrm{d}x \\
& \le \frac{p}{2} \int_0^\infty \int_0^\infty \left( x^m + y^m \right) K_{j,\varepsilon}(x,y) (f_{j,\varepsilon}(t,x+y))^{p-1} f_{j,\varepsilon}(t,x) f_{j,\varepsilon}(t,y)\ \mathrm{d}y\mathrm{d}x \\
& \quad - p \int_0^\infty \int_0^\infty x^m K_{j,\varepsilon}(x,y) (f_{j,\varepsilon}(t,x))^p f_{j,\varepsilon}(t,y)\ \mathrm{d}y\mathrm{d}x \\
& = p \int_0^\infty \int_0^\infty x^m K_{j,\varepsilon}(x,y) (f_{j,\varepsilon}(t,x+y))^{p-1} f_{j,\varepsilon}(t,x) f_{j,\varepsilon}(t,y)\ \mathrm{d}y\mathrm{d}x \\
& \quad - p \int_0^\infty \int_0^\infty x^m K_{j,\varepsilon}(x,y) (f_{j,\varepsilon}(t,x))^p f_{j,\varepsilon}(t,y)\ \mathrm{d}y\mathrm{d}x \\
\ .
\end{align*}
We next deduce from the convexity inequality
\begin{equation*}
p U^{p-1} V \le (p-1) U^p + V^p\ , \qquad (U,V)\in [0,\infty)^2\ ,
\end{equation*} that
\begin{align*}
P_{j,\varepsilon}(t) & \le (p-1) \int_0^\infty \int_0^\infty x^m K_{j,\varepsilon}(x,y) (f_{j,\varepsilon}(t,x+y))^p f_{j,\varepsilon}(t,y)\ \mathrm{d}y\mathrm{d}x \\
& \quad - (p-1) \int_0^\infty \int_0^\infty x^m K_{j,\varepsilon}(x,y) (f_{j,\varepsilon}(t,x))^p f_{j,\varepsilon}(t,y)\ \mathrm{d}y\mathrm{d}x \\
& \le (p-1) \int_0^\infty \int_y^\infty (x-y)^m K_{j,\varepsilon}(x-y,y) (f_{j,\varepsilon}(t,x))^p f_{j,\varepsilon}(t,y)\ \mathrm{d}x\mathrm{d}y \\
& \quad - (p-1) \int_0^\infty \int_y^\infty x^m K_{j,\varepsilon}(x,y) (f_{j,\varepsilon}(t,x))^p f_{j,\varepsilon}(t,y)\ \mathrm{d}x\mathrm{d}y \ .
\end{align*}
Now, the monotonicity of $x\mapsto x^m$ and $x\mapsto K_{j,\varepsilon}(x,y)$ implies that 
\begin{equation*}
(x-y)^m K_{j,\varepsilon}(x-y,y) \le x^m K_{j,\varepsilon}(x,y)\ , \qquad 0 < y < x \ .
\end{equation*} 
Consequently,
\begin{equation}
P_{j,\varepsilon}(t) \le 0 \ . \label{b89}
\end{equation}
Concerning the contribution of the fragmentation term, it reads 
\begin{align}
Q_{j,\varepsilon}(t) & := p \int_0^\infty x^m (f_{j,\varepsilon}(t,x))^{p-1} (\mathcal{F}_{j,\varepsilon}f_{j,\varepsilon})(t,x)\ \mathrm{d}x \nonumber \\
& = - p a_0 \Lambda_j(f_{j,\varepsilon}(t)) - p a_0 \varepsilon^2 L_{m,p}(f_{j,\varepsilon}(t)) + R_{j,\varepsilon}(t)\ , \label{b88} 
\end{align}
where
\begin{equation*}
\Lambda_j(f_{j,\varepsilon}(t)) := \int_0^\infty x^m \min\{x,j\}^\gamma (f_{j,\varepsilon}(t,x))^p\ \mathrm{d}x
\end{equation*}
and 
\begin{align*}
R_{j,\varepsilon}(t) & := p \int_0^\infty a_{j,\varepsilon}(y) f_{j,\varepsilon}(t,y) \int_0^y x^m b(x,y) (f_{j,\varepsilon}(t,x))^{p-1}\ \mathrm{d}x\mathrm{d}y \\
& = p a_0 \int_0^\infty \min\{y,j\}^\gamma y^{-1} f_{j,\varepsilon}(t,y) \int_0^y x^m B\left( xy^{-1} \right) (f_{j,\varepsilon}(t,x))^{p-1}\ \mathrm{d}x\mathrm{d}y \\
& \quad + p a_0 \varepsilon^2 \int_0^\infty y^{-1} f_{j,\varepsilon}(t,y) \int_0^y x^m B\left( xy^{-1} \right) (f_{j,\varepsilon}(t,x))^{p-1}\ \mathrm{d}x\mathrm{d}y\ .
\end{align*}
We infer from H\"older's inequality that 
\begin{align*}
& \int_0^y x^m B\left( xy^{-1} \right) (f_{j,\varepsilon}(t,x))^{p-1}\ \mathrm{d}x \\
& \qquad = \int_0^y x^{m/p} \min\{x,j\}^{-\gamma(p-1)/p} B\left( xy^{-1} \right) x^{m(p-1)/p} \min\{x,j\}^{\gamma(p-1)/p} (f_{j,\varepsilon}(t,x))^{p-1}\ \mathrm{d}x\\
& \qquad \le \left( \int_0^y x^m \min\{x,j\}^{-\gamma(p-1)} [B\left( x y^{-1} \right)]^p\ \mathrm{d}x \right)^{1/p} \left( \int_0^y x^m \min\{x,j\}^\gamma (f_{j,\varepsilon}(t,x))^p\ \mathrm{d}x \right)^{(p-1)/p} \ .
\end{align*}
Since
\begin{align*}
& \left( \int_0^y x^m \min\{x,j\}^{-\gamma(p-1)} [B\left( x y^{-1} \right)]^p\ \mathrm{d}x \right)^{1/p}\\
& \qquad \le \left( \int_0^y \left( x^{m-\gamma(p-1)} + x^m \right) [B\left( x y^{-1} \right)]^p\ \mathrm{d}x \right)^{1/p} \\
& \qquad \le y^{(m+1-\gamma (p-1))/p} \left( \int_0^1 z^{m-\gamma(p-1)} B(z)^p\ \mathrm{d}z \right)^{1/p} \\
& \qquad\quad + y^{(m+1)/p} \left( \int_0^1 z^m B(z)^p\ \mathrm{d}z \right)^{1/p} \ ,
\end{align*}
we further obtain
\begin{align*}
& \int_0^y x^m B\left( xy^{-1} \right) (f_{j,\varepsilon}(t,x))^{p-1}\ \mathrm{d}x \\
& \qquad \le  y^{(m+1-\gamma (p-1))/p} \left( \int_0^1 z^{m-\gamma(p-1)} B(z)^p\ \mathrm{d}z \right)^{1/p} \Lambda_j(f_{j,\varepsilon}(t))^{(p-1)/p} \\
& \qquad\quad + y^{(m+1)/p} \left( \int_0^1 z^m B(z)^p\ \mathrm{d}z \right)^{1/p} \Lambda_j(f_{j,\varepsilon}(t))^{(p-1)/p}\ .
\end{align*}
Similarly, by H\"older's inequality,
\begin{align*}
& \int_0^y x^m B\left( xy^{-1} \right) (f_{j,\varepsilon}(t,x))^{p-1}\ \mathrm{d}x \\
& \qquad = \int_0^y x^{m/p} B\left( xy^{-1} \right) x^{m(p-1)/p} (f_{j,\varepsilon}(t,x))^{p-1}\ \mathrm{d}x\\
& \qquad \le \left( \int_0^y x^{m} [B\left( x y^{-1} \right)]^p\ \mathrm{d}x \right)^{1/p} \left( \int_0^y x^{m} (f_{j,\varepsilon}(t,x))^p\ \mathrm{d}x \right)^{(p-1)/p} \\
& \qquad \le y^{(m+1)/p} \left( \int_0^1 z^{m} [B(z)]^p\ \mathrm{d}z \right)^{1/p} \left[ L_{m,p}(f_{j,\varepsilon}(t)) \right]^{(p-1)/p} \ .
\end{align*}
Since $0\le m-\gamma(p-1)\le m$ and $p\in [1,p_0]$ by \eqref{b14}, we infer from \eqref{b95a} that 
\begin{equation*}
\int_0^1 z^{m} B(z)^p\ \mathrm{d}z \le \int_0^1 z^{m-\gamma(p-1)} B(z)^p\ \mathrm{d}z \le \mathcal{B}_p^p < \infty\ .
\end{equation*}
Gathering the above estimates and using Young's inequality, we end up with
\begin{align}
R_{j,\varepsilon}(t) & \le p a_0 \mathcal{B}_p M_{(m+1+\gamma-p)/p}(f_{j,\varepsilon}(t)) \Lambda_j(f_{j,\varepsilon}(t))^{(p-1)/p} \nonumber \\ 
& \qquad + p a_0 \mathcal{B}_p M_{(m+1+\gamma p-p)/p}(f_{j,\varepsilon}(t)) \Lambda_j(f_{j,\varepsilon}(t))^{(p-1)/p} \nonumber \\ 
& \qquad + p a_0 \varepsilon^2 \mathcal{B}_p M_{(m+1-p)/p}(f_{j,\varepsilon}(t)) \left[ L_{m,p}(f_{j,\varepsilon}(t)) \right]^{(p-1)/p} \nonumber \\ 
& \le \frac{p-1}{2} a_0 \Lambda_j(f_{j,\varepsilon}(t)) + 2^{p-1} a_0 \mathcal{B}_p^p \mathcal{M}_{(m+1+\gamma-p)/p,j,\varepsilon}^p \nonumber \\
& \qquad + \frac{p-1}{2} a_0 \Lambda_j(f_{j,\varepsilon}(t)) + 2^{p-1} a_0 \mathcal{B}_p^p \mathcal{M}_{(m+1+\gamma p-p)/p,j,\varepsilon}^p \nonumber \\
& \qquad + (p-1) a_0 \varepsilon^2 L_{m,p}(f_{j,\varepsilon}(t)) + a_0 \varepsilon^2 \mathcal{B}_p^p \mathcal{M}_{(m+1-p)/p,j,\varepsilon}^p \ . \label{b87}
\end{align}
We then deduce from \eqref{b88} and \eqref{b87} that
\begin{equation}
Q_{j,\varepsilon}(t) \le - a_0  \left[ \Lambda_j(f_{j,\varepsilon}(t)) + \varepsilon^2 L_{m,p}(f_{j,\varepsilon}(t)) \right] + a_0 S_{j,\varepsilon}(m,p)\ . \label{b86}
\end{equation}
Combining \eqref{b70}, \eqref{b89}, and \eqref{b86} leads us to the differential inequality
\begin{equation}
\frac{d}{dt} L_{m,p}(f_{j,\varepsilon}(t)) + a_0 \left[ \Lambda_j(f_{j,\varepsilon}(t)) + \varepsilon^2 L_{m,p}(f_{j,\varepsilon}(t)) \right] \le a_0 S_{j,\varepsilon}(m,p) \label{b85}
\end{equation} 
for $t>0$. We first infer from \eqref{b85} that, for $t>0$, 
\begin{equation*}
\frac{d}{dt} L_{m,p}(f_{j,\varepsilon}(t)) + a_0 \varepsilon^2 L_{m,p}(f_{j,\varepsilon}(t)) \le a_0 S_{j,\varepsilon}(m,p)\ .
\end{equation*}
Hence, after integration with respect to time,
\begin{align*}
\int_0^\infty L_{m,p}(f_{j,\varepsilon}(t)) & \le e^{-a_0 \varepsilon^2 t} L_{m,p}(f^{in}) + \frac{S_{j,\varepsilon}(m,p)}{\varepsilon^2} \left( 1 - e^{-a_0 \varepsilon^2 t} \right) \\
& \le \max\left\{ L_{m,p}(f^{in}) , \frac{S_{j,\varepsilon}(m,p)}{\varepsilon^2} \right\}\ .
\end{align*}
from which \eqref{b15} follows. We also infer from \eqref{b85} that, for $t>0$,
\begin{equation*}
\frac{d}{dt} L_{m,p}(f_{j,\varepsilon}(t)) + a_0 \Lambda_j(f_{j,\varepsilon}(t)) \le a_0 S_{j,\varepsilon}(m,p)\ .
\end{equation*} 
Integrating with respect to time and using the non-negativity of $L_{m,p}(f_{j,\varepsilon}(t))$, we obtain
\begin{equation*}
a_0 \int_0^t \Lambda_j(f_{j,\varepsilon}(s))\ \mathrm{d}s \le L_{m,p}(f^{in}) + a_0 t S_{j,\varepsilon}(m,p)
\end{equation*}
for $t> 0$. Dividing the above inequality by $a_0 t$ gives \eqref{b16}.
\end{proof}

Combining the outcome of Lemma~\ref{lemb30} and Lemma~\ref{lemb3b} leads to an $\varepsilon$-dependent $L^p$-estimate for $(f_{j,\varepsilon})_{j\ge 2}$ for a suitable value of $p$.

\begin{corollary}\label{corb31}
Let $m_0\in (m_\star,0)$, $m_1\in (\lambda,1)$, and $p_1\in (1,p_0)$ be such that
\begin{equation}
1 < p_1 < \frac{m_1+1}{\lambda+1} \;\text{ and }\; p_1 \le\frac{m_1+\gamma}{\gamma}\ . \label{b140}
\end{equation}
For $\varepsilon\in (0,\varepsilon_{m_0,\sigma})$ and $t\ge 0$,
\begin{align*}
L_{0,p_2}(f_{j,\varepsilon}(t)) & \le \max\left\{ M_{m_0}(f^{in}) , \mu_{m_0} \sigma^2 \varepsilon^{-(\gamma+2-2m_0)/\gamma} \right\} \\
& \qquad + \max\left\{ L_{m_1,p_1}(f^{in}) , \frac{S_{j,\varepsilon}(m_1,p_1)}{\varepsilon^2} \right\}\ ,
\end{align*} 
where
\begin{equation*}
p_2 := \frac{m_1}{m_1-m_0} + p_1 \frac{|m_0|}{m_1-m_0} \in (1,p_1)\ .
\end{equation*}
\end{corollary}

\begin{proof}
Since 
\begin{equation*}
\frac{m_1}{m_1-m_0} m_0 + \frac{|m_0|}{m_1-m_0} m_1 = 0\ ,
\end{equation*}
we infer from Young's inequality that, if $h\in X_{m_0}\cap L^{p_1}((0,\infty),x^{m_1}\mathrm{d}x)$, then $h\in L^{p_2}(0,\infty)$ and
\begin{align}
L_{0,p_2}(h) = \|h\|_{p_2}^{p_2} & = \int_0^\infty \left( x^{m_0} |h(x)| \right)^{m_1/(m_1-m_0)} \left( x^{m_1} |h(x)| \right)^{|m_0|/(m_1-m_0)}\ \mathrm{d}x \nonumber \\
& \le \frac{m_1}{m_1-m_0} \int_0^\infty x^{m_0} |h(x)|\ \mathrm{d}x + \frac{|m_0|}{m_1-m_0} \int_0^\infty x^{m_1} |h(x)|^{p_1}\ \mathrm{d}x \nonumber \\
& \le M_{m_0}(|h|) + L_{m_1,p_1}(h)\ . \label{zorglub}
\end{align} 
Now, consider $t\ge 0$. As $\varepsilon\in (0,\varepsilon_{m_0,\sigma})$ and $p_1$ satisfies \eqref{b140}, Corollary~\ref{corb31} readily follows from Lemma~\ref{lemb30} (with $m=m_0$), Lemma~\ref{lemb3b} (with $(m,p)=(m_1,p_1)$), and \eqref{zorglub} (with $h=f_{j,\varepsilon}(t)$).
\end{proof}

\subsection{Time Equicontinuity}\label{sec2.4}

The last estimate to be derived in this section provides the time equicontinuity of the sequence $(f_{j,\varepsilon})_{j\ge 2}$ in $L^1(0,\infty)$, which is needed later to apply a variant of the Arzel\`a-Ascoli theorem.

\begin{lemma}\label{lemb6}
\refstepcounter{NumConstB}\label{cstB6} There is a positive constant $C_{\ref{cstB6}}>0$ such that
\begin{equation*}
\|\partial_t f_{j,\varepsilon}(t)\|_1 \le C_{\ref{cstB6}} \left( \sigma + \mathcal{M}_{0,j,\varepsilon}^2 \right) \ , \qquad t\ge 0\ .
\end{equation*} 
\end{lemma}

\begin{proof}
Let $t>0$. It follows from \eqref{b70a}, \eqref{b950}, and Fubini's theorem that 
\begin{align*}
\|\partial_t f_{j,\varepsilon}(t)\|_1 & \le \frac{3}{2} \int_0^\infty \int_0^\infty K_{j,\varepsilon}(x,y) f_{j,\varepsilon}(t,x) f_{j,\varepsilon}(t,y)\ \mathrm{d}y\mathrm{d}x \\
& \qquad + (1 + \mathfrak{b}_0) \int_0^\infty a_{j,\varepsilon}(x) f_{j,\varepsilon}(t,x)\ \mathrm{d}x \\
& \le \frac{3K_0}{2} \int_0^\infty \int_0^\infty \left( x^\lambda + y^\lambda + 2\varepsilon \right) f_{j,\varepsilon}(t,x) f_{j,\varepsilon}(t,y)\ \mathrm{d}y\mathrm{d}x \\
& \qquad + a_0 (1+\mathfrak{b}_0) \int_0^\infty \left( x^\gamma + \varepsilon^2 \right) f_{j,\varepsilon}(t,x)\ \mathrm{d}x \\
& \le 3K_0 \left[ M_\lambda(f_{j,\varepsilon}(t)) M_0(f_{j,\varepsilon}(t)) + M_0(f_{j,\varepsilon}(t))^2 \right] \\
& \qquad + a_0 (1+\mathfrak{b}_0) \left[ M_\gamma(f_{j,\varepsilon}(t)) + M_0(f_{j,\varepsilon}(t)) \right]\ .
\end{align*}
We then infer from \eqref{b4}, \eqref{b98a}, \eqref{b97}, and the inequalities
\begin{equation*}
x^\lambda \le 1 + x \ , \qquad x^\gamma \le 1 + x^{1+\gamma}\ , \qquad x\ge 0\ ,
\end{equation*} 
that
\begin{align*}
\| \partial_t f_{j,\varepsilon}(t)\|_1 & \le 3K_0 \left[ M_0(f_{j,\varepsilon}(t)) M_1(f_{j,\varepsilon}(t)) + 2 M_0(f_{j,\varepsilon}(t))^2 \right] \\
& \qquad + a_0 (1+\mathfrak{b}_0) \left[ M_{\gamma+1}(f_{j,\varepsilon}(t)) + 2M_0(f_{j,\varepsilon}(t)) \right] \\
& \le 3K_0 \left[ \varrho^2 + 3 \mathcal{M}_{0,j,\varepsilon}^2 \right] + a_0 (1+\mathfrak{b}_0) \left( \sigma + 1 +  \mathcal{M}_{0,j,\varepsilon}^2 \right) \\
& \le [3K_0(\varrho+3) + 2 a_0 (1+\mathfrak{b}_0)] \left( \sigma + \mathcal{M}_{0,j,\varepsilon}^2 \right)\ ,
\end{align*}
and the proof is complete.
\end{proof}

\section{Stationary solutions by a dynamical approach: $\varepsilon\in (0,1)$}\label{sec3}

In this section, we fix $\varepsilon\in (0,1)$ and study the coagulation-fragmentation equation \eqref{CFe} with coagulation kernel $K_\varepsilon$ and overall fragmentation rate $a_\varepsilon$ given by 
\begin{equation}
K_\varepsilon(x,y) = K(x,y) + 2\varepsilon K_0\ , \qquad a_\varepsilon(x) = a(x) + a_0\varepsilon^2\ , \qquad (x,y)\in (0,\infty)^2\ ; \label{c2}
\end{equation}
that is, 
\begin{subequations}\label{c1}
\begin{align}	
\partial_t f & = \mathcal{C}_\varepsilon f + \mathcal{F}_\varepsilon f\ , \qquad (t,x) \in (0,\infty)^2\ , \label{c1a} \\
f(0) & = f^{in}\ , \qquad x\in (0,\infty)\ , \label{c1b}
\end{align}
\end{subequations}
where the coagulation and fragmentation operators $\mathcal{C}_\varepsilon$ and $\mathcal{F}_\varepsilon$ are defined in \eqref{CFe}. 

Several results are established in this section. We begin with the well-posedness of \eqref{c1} for a suitable class of initial conditions, the existence of solutions being obtained by passing to the limit as $j\to\infty$ in \eqref{b70} (Section~\ref{sec3.1}). We also establish the continuity of the solutions to \eqref{c1} with respect to the initial condition for the weak topology of $X_1$ (Section~\ref{sec3.3}) and construct an invariant set for the dynamics of \eqref{c1} (Section~\ref{sec3.4}). Combining the outcome of this analysis with a consequence of Tychonov's fixed point theorem provides the existence of a stationary solution to \eqref{c1a} (Section~\ref{sec3.5}). The estimates derived in the previous section are of course at the heart of the proofs of the results of this section.

\medskip

We fix 
\begin{subequations}\label{c3}
\begin{equation}
m_0\in (m_\star,0)\cap (-1,0)\ , \qquad m_1\in (\lambda,1)\ , \qquad p_1\in (1,p_0)\ , \label{c3a}
\end{equation}
such that
\begin{equation}
1 < p_1 < \frac{m_1+1}{\lambda+1} \;\text{ and }\; p_1 \le \frac{m_1+\gamma}{\gamma}\ . \label{c3c}
\end{equation}
\end{subequations}
We recall that \eqref{c3} implies that
\begin{subequations}\label{c4}
\begin{equation}
m_2\in (\lambda,1) \;\text{ and }\; m_2 < \frac{m_1+1+\gamma-p_1}{p_1} < \frac{m_1+1+\gamma p_1-p_1}{p_1} \le 1+\gamma \ , \label{c4a}
\end{equation}
where
\begin{equation}
m_2 := \frac{m_1+1-p_1}{p_1} < 1\ . \label{c4b}
\end{equation}
\end{subequations}
We also fix $\varrho>0$ and $\sigma>0$ satisfying 
\begin{equation}
\sigma >\max\left\{ 1 , \varrho, \mu_2 , \mu_{2+\gamma} \right\}\ , \label{c5}
\end{equation} 
recalling that $\mu_m$ is defined in Lemma~\ref{lemb1} for $m\ge 2$. 

We next define a subset $\mathcal{Y}_\varepsilon$ of $X_1^+$ as follows: $h\in \mathcal{Y}_\varepsilon$ if and only if 
\begin{subequations}\label{c300}
\begin{align}
h \in X_1^+ \cap X_{m_0} \cap X_{2+\gamma}\ , & \qquad M_1(h) = \varrho\ , \label{c300a}\\
\max\{M_2(h) , M_{2+\gamma}(h)\} \le \sigma\ , & \qquad M_{m_2}(h) \le \mu_{m_2} + \sigma\ , \label{c300b}\\
M_0(h) \le \sigma + \mu_0 \varepsilon^{-1}\ , & \qquad M_{m_0}(h) \le \mu_{m_0} \sigma^2 \varepsilon^{-(\gamma+2-2m_0)/\gamma}\ , \label{c300c}\\
L_{m_1,p_1}(h) \le \sigma_1 \varepsilon^{-2}\ , \label{c300d}
\end{align}
\end{subequations}
where
\begin{equation}
\sigma_1 := 2^{p_1} \mathcal{B}_{p_1}^{p_1} \left[ 2 \sigma^{p_1} + 3 (\mu_{m_2}+\sigma)^{p_1} \right] \label{c301}
\end{equation}
and $\mathcal{B}_{p_1}$ is defined in \eqref{b95a}.

\subsection{Well-posedness of \eqref{c1}}\label{sec3.1}

We begin with the well-posedness of \eqref{c1} in $\mathcal{Y}_\varepsilon$, along with several estimates for its solutions. 

\begin{proposition}\label{propc1}
Consider $\varepsilon\in (0,\varepsilon_{m_0,\sigma})$ and $f^{in}\in \mathcal{Y}_\varepsilon$, recalling that
\begin{equation*}
\varepsilon_{m_0,\sigma} = \frac{1}{\sigma} \min\left\{ 1 , \frac{K_0 \varrho^2}{4 a_0 \mathfrak{b}_{m_0}} \right\}
\end{equation*}
is defined in \eqref{champignac} with $\mathfrak{b}_{m_0}$ given by \eqref{b95a}. There is a unique weak solution 
\begin{equation*}
\Psi_\varepsilon(\cdot,f^{in}) = f_\varepsilon \in C([0,\infty),X_0^+)\cap C([0,\infty),X_{1,w})
\end{equation*}
to \eqref{c1} which satisfies 
\begin{equation}
\begin{split}
\frac{d}{dt} \int_0^\infty \vartheta(x) f_\varepsilon(t,x)\ \mathrm{d}x & = \frac{1}{2} \int_0^\infty \int_0^\infty K_\varepsilon(x,y) \chi_\vartheta(x,y) f_\varepsilon(t,x) f_\varepsilon(t,y)\ \mathrm{d}y \mathrm{d}x \\
& \qquad - \int_0^\infty a_\varepsilon(y) N_\vartheta(y) f_\varepsilon(t,y)\ \mathrm{d}y \ ,
\end{split} \label{c302}
\end{equation}
for all $t\ge 0$ and $\vartheta\in L^\infty(0,\infty)$, the functions $\chi_\vartheta$ and $N_\vartheta$ being defined in \eqref{b6}, and possesses the following properties:
\begin{subequations}\label{c6}
\begin{align}
M_1(f_\varepsilon(t)) & = \varrho\ , \qquad t\ge 0\ , \label{c6a} \\
\sup_{t\ge 0} M_m(f_\varepsilon(t)) & \le \sigma\ , \qquad m\in (1,2+\gamma]\ , \label{c6b} \\
\sup_{t\ge 0} M_m(f_\varepsilon(t)) & \le \max\left\{ M_m(f^{in}) , \sigma + \mu_m \right\}\ , \qquad m\in (\lambda,1)\ , \label{c6c}
\end{align}
\end{subequations}
\begin{subequations}\label{c7}
\begin{align}
\sup_{t\ge 0} M_0(f_\varepsilon(t)) & \le \sigma + \mu_0 \varepsilon^{-1}\ , \label{c7a} \\
\sup_{t\ge 0} M_{m_0}(f_\varepsilon(t)) & \le \mu_{m_0} \sigma^2 \varepsilon^{-(\gamma+2-2m_0)/\gamma}\ , \label{c7b} 
\end{align}
\end{subequations}
\begin{subequations}\label{c8}
\begin{align}
\sup_{t\ge 0} L_{m_1,p_1}(f_\varepsilon(t)) & \le \sigma_1 \varepsilon^{-2}\ , \label{c8a} \\
\sup_{t\ge 0} L_{0,p_2}(f_\varepsilon(t)) & \le \kappa_\varepsilon := \mu_{m_0} \sigma^2 \varepsilon^{-(\gamma+2-2m_0)/\gamma} + \sigma_1 \varepsilon^{-2} \ , \label{c8b} 
\end{align}
\end{subequations}
and
\begin{equation}
\frac{1}{t} \int_0^t L_{m_1+\gamma,p_1}(f_\varepsilon(s))\ \mathrm{d}s \le \frac{1}{a_0 t} L_{m_1,p_1}(f^{in}) + \sigma_1\ , \qquad t>0\ . \label{c9}
\end{equation}
Moreover, if $f^{in}\in X_m$ for some $m>2+\gamma$, then $f_\varepsilon\in L^\infty((0,\infty),X_m)$ and
\begin{equation}
\sup_{t\ge 0} M_m(f_\varepsilon(t)) \le \max\{ M_m(f^{in}) , \mu_m \}\ , \label{c10}
\end{equation}
the constant $\mu_m$ being defined in Lemma~\ref{lemb1}.
\end{proposition}

\begin{proof} \textbf{Step~1: Existence.} Let $j\ge 2$ and recall that $f_{j,\varepsilon}$ is the strong solution to the coagulation-fragmentation equation \eqref{b70}, see Section~\ref{sec2}. Since $f^{in}\in \mathcal{Y}_\varepsilon$, it follows from \eqref{b4} that
\begin{equation}
M_1(f_{j,\varepsilon}(t)) = \varrho\ , \qquad t\ge 0\ ,	\ j\ge 2\ ,\label{c11}
\end{equation}
and from \eqref{b98}, \eqref{c5}, Lemma~\ref{lemb1}, and Corollary~\ref{corb1b} that
\begin{equation}
\sup_{t\ge 0} M_m(f_{j,\varepsilon}(t)) \le \sigma\ , \qquad m\in (1,2+\gamma]\ , \ j\ge 2\ . \label{c12}
\end{equation}
Next, \eqref{b98}, \eqref{c5}, \eqref{c300b}, and Lemma~\ref{lemb2} guarantee that
\begin{equation}
\sup_{t\ge 0} M_{m_2}(f_{j,\varepsilon}(t)) \le \mu_{m_2}+\sigma\ , \qquad j\ge 2\ , \label{c14}
\end{equation}
while, since $\varepsilon\in (0,\varepsilon_{m_0,\sigma})$, we deduce from \eqref{b98}, \eqref{c5}, \eqref{c300c}, Lemma~\ref{lemb3}, and Lemma~\ref{lemb30} that 
\begin{align}
\sup_{t\ge 0} M_0(f_{j,\varepsilon}(t)) & \le \sigma + \mu_0 \varepsilon^{-1}\ , \qquad j\ge 2\ , \label{c15} \\
\sup_{t\ge 0} M_{m_0}(f_{j,\varepsilon}(t)) & \le \mu_{m_0} \sigma^2 \varepsilon^{-(\gamma+2-2m_0)/\gamma}\ , \qquad j\ge 2\ . \label{c16}
\end{align}
Finally, by \eqref{c4}, \eqref{c12}, and H\"older's and Young's inequalities,
\begin{align*}
M_{(m+1+\gamma-p_1)/p_1}(f_{j,\varepsilon}(t))^{p_1} & \le \frac{\gamma}{p_1(1+\gamma-m_2)} M_{1+\gamma}(f_{j,\varepsilon}(t))^{p_1} \\
& \qquad + \frac{p_1(1+\gamma-m_2)-\gamma}{p_1(1+\gamma-m_2)} M_{m_2}(f_{j,\varepsilon}(t))^{p_1} \\
& \le \sigma^{p_1} + M_{m_2}(f_{j,\varepsilon}(t))^{p_1}\ ,
\end{align*}
and
\begin{align*}
M_{(m+1+\gamma p_1-p_1)/p_1}(f_{j,\varepsilon}(t))^{p_1} & \le \frac{\gamma}{1+\gamma-m_2} M_{1+\gamma}(f_{j,\varepsilon}(t))^{p_1} + \frac{1-m_2}{1+\gamma-m_2} M_{m_2}(f_{j,\varepsilon}(t))^{p_1} \\
& \le \sigma^{p_1} + M_{m_2}(f_{j,\varepsilon}(t))^{p_1}
\end{align*}
for $t\ge 0$ and $j\ge 2$, so that, using also \eqref{c301} and \eqref{c14}, 
\begin{align}
S_{j,\varepsilon}(m_1,p_1) & = 2^{p_1} \mathcal{B}_{p_1}^{p_1} \sup_{t\ge 0} M_{(m+1+\gamma-p_1)/p_1}(f_{j,\varepsilon}(t))^{p_1} \nonumber\\
& \qquad + 2^{p_1} \mathcal{B}_{p_1}^{p_1} \sup_{t\ge 0} M_{(m+1+\gamma p_1-p_1)/p_1}(f_{j,\varepsilon}(t))^{p_1} \nonumber\\
& \qquad + 2^{p_1} \mathcal{B}_{p_1}^{p_1} \varepsilon^2 \sup_{t\ge 0}M_{m_2}(f_{j,\varepsilon}(t))^{p_1} \nonumber\\
& \le 2^{p_1} \mathcal{B}_{p_1}^{p_1} \left[ 2 \sigma^{p_1} + 3 \left( \mu_{m_2} + \sigma \right)^{p_1}\right] = \sigma_1 \ . \label{c17}
\end{align} 
Combining \eqref{b15}, \eqref{c300d}, and \eqref{c17}, we conclude that
\begin{equation}
\sup_{t\ge 0} L_{m_1,p_1}(f_{j,\varepsilon}(t)) \le \sigma_1 \varepsilon^{-2}\ , \qquad j\ge 2\ . \label{c17.5}
\end{equation}
A straightforward consequence of \eqref{c3c}, \eqref{c5}, \eqref{c6}, \eqref{c16}, \eqref{c17.5}, and Corollary~\ref{corb31} is the bound
\begin{equation}
\sup_{t\ge 0} L_{0,p_2}(f_{j,\varepsilon}(t)) \le \mu_{m_0} \sigma^2 \varepsilon^{-(\gamma+2-2m_0)/\gamma} + \sigma_1 \varepsilon^{-2} = \kappa_\varepsilon\ , \qquad j\ge 2\ . \label{c18}
\end{equation} 

Now, introducing the set 
\begin{equation}
\mathcal{W}_\varepsilon := \left\{ 
\begin{array}{c}
h \in X_{m_0}\cap X_{2+\gamma}\cap L^{p_2}(0,\infty)\ : \\
 \\
M_{2+\gamma}(h)\le \sigma\ , \ \max\{M_{m_0}(h) , L_{0,p_2}(h)\} \le \kappa_\varepsilon
\end{array}
\right\}\ , \label{c60}
\end{equation}
it readily follows from \eqref{c12}, \eqref{c16}, and \eqref{c18} that 
\begin{equation}
f_{j,\varepsilon}(t)\in \mathcal{W}_\varepsilon\ , \qquad t\ge 0\ , \ j\ge 2\ , \label{c19}
\end{equation}
while the Dunford-Pettis theorem ensures that 
\begin{equation}
\mathcal{W}_\varepsilon \,\text{ is a relatively sequentially weakly compact subset of }\, X_m \label{c61}
\end{equation}
for any $m\in (m_0,2+\gamma)$, and in particular of $X_0$. Moreover, it follows from \eqref{c15} and Lemma~\ref{lemb6} that, for $0 \le t_1 \le t_2$ and $j\ge 2$, 
\begin{equation}
\|f_{j,\varepsilon}(t_2) - f_{j,\varepsilon}(t_1)\|_1 \le \int_{t_1}^{t_2} \|\partial_t f_{j,\varepsilon}(t)\|_1\ \mathrm{d}t \le C_{\ref{cstB6}} \left[ \sigma + \left( \sigma + \mu_0 \varepsilon^{-1} \right)^2 \right] (t_2-t_1) \ . \label{c20}
\end{equation}
Consequently, $(f_{j,\varepsilon})_{j\ge 2}$ is equicontinuous at each $t\ge 0$ for the norm-topology of $L^1(0,\infty)$, and thus it is also equicontinuous for the weak topology of $L^1(0,\infty)$. This property, along with \eqref{c19} and the relative compactness \eqref{c61} of $\mathcal{W}_\varepsilon$, allows us to apply a variant of the Arzel\`a-Ascoli theorem \cite[Theorem~A.3.1]{Vrab03} to conclude that there are a subsequence of $(f_{j,\varepsilon})_{j\ge 2}$ (possibly depending on $\varepsilon$ but not relabeled) and $f_\varepsilon\in C([0,\infty),X_{0,w})$ such that
\begin{equation}
f_{j,\varepsilon} \longrightarrow f_\varepsilon \;\text{ in }\; C([0,T],X_{0,w}) \;\text{ for all }\; T>0\ . \label{c21}
\end{equation} 
A first consequence of \eqref{c21} is that $f_\varepsilon(t)\in X_0^+$ for all $t\ge 0$. It next follows from \eqref{c11}, \eqref{c12}, \eqref{c15}, \eqref{c16}, and \eqref{c21} by a weak lower semicontinuity argument that $f_\varepsilon$ satisfies \eqref{c6b}, \eqref{c7a}, \eqref{c7b}, and 
\begin{equation*}
M_1(f_\varepsilon(t)) \le \varrho\ , \qquad t\ge 0\ .
\end{equation*} 
A similar argument allows us to deduce \eqref{c6c} from Lemma~\ref{lemb2} and \eqref{c21}. We then combine the just established property \eqref{c6b} with \eqref{c12} and \eqref{c21} to improve the convergence \eqref{c21} to 
\begin{equation}
f_{j,\varepsilon} \longrightarrow f_\varepsilon \;\text{ in }\; C([0,T],X_{1,w}\cap X_{\gamma,w}) \;\text{ for all }\; T>0\ . \label{c22}
\end{equation}  
Recalling \eqref{c11}, we readily infer from \eqref{c22} that $f_\varepsilon$ satisfies the mass conservation \eqref{c6a}. We employ again weak lower semicontinuity arguments to deduce \eqref{c8} and
\begin{equation}
\frac{1}{t} \int_0^t \int_0^R x^{m+\gamma} (f_\varepsilon(s,x))^p\ \mathrm{d}x\mathrm{d}s \le \frac{1}{a_0 t} L_{m_1,p_1}(f^{in}) + \sigma_1\ , \qquad t>0\ , \  R\ge 1\ , \label{c200}
\end{equation}
from \eqref{b15}, \eqref{b16}, \eqref{c3c}, \eqref{c300d}, \eqref{c17}, \eqref{c17.5}, \eqref{c18}, and \eqref{c21}. As the right-hand side of \eqref{c200} does not depend on $R$, we may let $R\to\infty$ in \eqref{c200} and use Fatou's lemma to obtain \eqref{c9}. 

Now, owing to \eqref{a14}, \eqref{a15}, \eqref{c21}, and \eqref{c22}, we may proceed as in \cite{Stew89}, see also \cite{BLLxx, ELMP03, EMRR05, LaMi02b}, to deduce from \eqref{b5} that $f_\varepsilon$ is a weak solution to \eqref{c1}, in the sense that it satisfies \eqref{c302}. Furthermore, we may argue as in the proof of Lemma~\ref{lemb6} with the help of \eqref{c6a}, \eqref{c6b}, and \eqref{c7a} to show that $\partial_t f_\varepsilon(t)$ belongs to $X_0$ for any $t\ge 0$ and satisfies
\begin{equation}
\|\partial_t f_\varepsilon(t)\|_1 \le C_{\ref{cstB6}} \left[ \sigma + \left( \sigma + \mu_0 \varepsilon^{-1} \right)^2 \right]\ , \qquad t\ge 0\ , \label{c62}
\end{equation} 
the constant $C_{\ref{cstB6}}$ being defined in Lemma~\ref{lemb6}.
\smallskip

\noindent\textbf{Step~2: Uniqueness.} It is a consequence of \cite[Theorem~8.2.55]{BLLxx} (with $\ell(x) = 1 + x^{\max\{1,\gamma\}}$, $x>0$, and $\zeta=1$), see also \cite{EMRR05}.

\smallskip

\noindent\textbf{Step~3: Higher moments.} Finally, if $f^{in}\in X_m$ for some $m>2+\gamma$, then the proof of \eqref{c10} relies on a weak lower semicontinuity argument as that of \eqref{c6b} and follows from \eqref{c21} and Lemma~\ref{lemb1}.
\end{proof}

\subsection{Invariant Set}\label{sec3.4}

As a consequence of the various estimates derived in Proposition~\ref{propc1}, we construct a subset $\mathcal{Z}_\varepsilon$ of $\mathcal{Y}_\varepsilon$ which is left invariant by $\Psi_\varepsilon$. Specifically, $h\in \mathcal{Z}_\varepsilon$ if and only if 
\begin{subequations}\label{c35}
\begin{align}
& h \in \mathcal{Y}_\varepsilon \cap \bigcap_{m>2+\gamma} X_m\ , \label{c35a}\\
& \qquad M_m(h) \le \mu_m\ , \qquad m>2+\gamma\ , \label{c35b}\\
& \qquad M_m(h) \le \sigma\ , \qquad m\in (1,2+\gamma]\ , \label{c35c}\\
& \qquad M_m(h) \le \sigma + \mu_m\ , \qquad m\in (\lambda,1)\ . \label{c35d}
	\end{align}
\end{subequations}

\begin{proposition}\label{propc3}
Consider $\varepsilon\in (0,\varepsilon_{m_0,\sigma})$ and $f^{in}\in \mathcal{Z}_\varepsilon$. Then $\Psi_\varepsilon(t,f^{in})\in \mathcal{Z}_\varepsilon$ for all $t\ge 0$.
\end{proposition}

\begin{proof}
Set $f_\varepsilon := \Psi_\varepsilon(\cdot,f^{in})$ and consider $t>0$. We first deduce from \eqref{c6a}, \eqref{c6b} (with $m=2$ and $m=2+\gamma$), \eqref{c6c} (with $m=m_2$), \eqref{c7}, and \eqref{c8a} that $f_\varepsilon(t)\in \mathcal{Y}_\varepsilon$. In addition, $f_\varepsilon(t)\in X_m$ for all $m>2+\gamma$ and satisfies \eqref{c35b} by \eqref{c10}, while \eqref{c35c} and \eqref{c35d} follow from \eqref{c6b} and \eqref{c6c}, respectively.
\end{proof}

\subsection{Dynamical System in $X_{1,w}$}\label{sec3.3}

We go on with the continuity properties of the map $f^{in}\mapsto \Psi_\varepsilon(.,f^{in})$ defined in Proposition~\ref{propc1} and actually show that $\Psi_\varepsilon$ is a dynamical system on $\mathcal{Y}_\varepsilon$ for the weak topology of $X_1$. 

\begin{proposition}\label{propc2}
Consider $\varepsilon\in (0,\varepsilon_{m_0,\sigma})$, $f^{in}\in \mathcal{Y}_\varepsilon$, and a sequence $(f_n^{in})_{n\ge 1}$ of initial conditions in $\mathcal{Y}_\varepsilon$ such that
\begin{equation}
f_n^{in} \rightharpoonup f^{in} \;\text{ in }\; X_1\ . \label{c30}
\end{equation}
Then, for any $T>0$,
\begin{equation*}
\Psi_\varepsilon(\cdot,f_n^{in}) \longrightarrow \Psi_\varepsilon(\cdot,f^{in}) \;\text{ in }\; C([0,T],X_{1,w})\ .
\end{equation*} 
\end{proposition}

\begin{proof}
For $n\ge 1$ we put $f_{\varepsilon,n}:= \Psi_\varepsilon(\cdot,f_n^{in})$. On the one hand, it follows from \eqref{c6b}, \eqref{c7b}, and \eqref{c8b} that
\begin{equation}
f_{\varepsilon,n}(t) \in \mathcal{W}_\varepsilon \ , \qquad t\ge 0\ , \ n\ge 1\ , \label{c63}
\end{equation}
recalling that the set $\mathcal{W}_\varepsilon$ is defined in \eqref{c60}. On the other hand, let $0\le t_1 < t_2$ and $n\ge 1$. We infer from \eqref{c62} that
\begin{equation*}
\|f_{\varepsilon,n}(t_2)-f_{\varepsilon,n}(t_1)\|_1 \le C_{\ref{cstB6}} \left[ \sigma + \left( \sigma + \mu_0 \varepsilon^{-1} \right)^2 \right] (t_2-t_1)\ .
\end{equation*} 
Combining this estimate with \eqref{c6b} gives, for $R>0$, 
\begin{align*}
\int_0^\infty x \left| f_{\varepsilon,n}(t_2,x) - f_{\varepsilon,n}(t_1,x) \right|\ \mathrm{d}x & \le R \|f_{\varepsilon,n}(t_2)-f_{\varepsilon,n}(t_1)\|_1 \\
& \qquad + \frac{1}{R} \int_R^\infty x^2 \left( f_{\varepsilon,n}(t_2,x) + f_{\varepsilon,n}(t_1,x) \right)\ \mathrm{d}x \\
& \le C_{\ref{cstB6}} \left[ \sigma + \left( \sigma + \mu_0 \varepsilon^{-1} \right)^2 \right] (t_2-t_1) + \frac{2\sigma}{R}\ .
\end{align*}
Now, taking $R=1/\sqrt{t_2-t_1}$ in the previous inequality, we end up with
\begin{equation*}
\int_0^\infty x \left| f_{\varepsilon,n}(t_2,x) - f_{\varepsilon,n}(t_1,x) \right|\ \mathrm{d}x \le \left\{ C_{\ref{cstB6}} \left[ \sigma + \left( \sigma + \mu_0 \varepsilon^{-1} \right)^2 \right] + 2 \sigma \right\} \sqrt{t_2-t_1}\ . 
\end{equation*}
Consequently, the sequence $(f_{\varepsilon,n})_{n\ge 1}$ is equicontinuous at each $t\ge 0$ for the norm-topology of $X_1$ and thus also for the weak topology of $X_1$. Recalling \eqref{c61} and \eqref{c63}, we are again in a position to use the variant of the Arzel\`a-Ascoli theorem stated in \cite[Theorem~A.3.1]{Vrab03} to deduce that there are $F_\varepsilon\in C([0,\infty),X_{1,w})$ and a subsequence $(f_{\varepsilon,n_k})_{k\ge 1}$ of $(f_{\varepsilon,n})_{n\ge 1}$ (possibly depending on $\varepsilon$) such that
\begin{equation}
f_{\varepsilon,n_k} \longrightarrow F_\varepsilon \;\text{ in }\; C([0,T],X_{1,w}) \label{c32}
\end{equation}  
for any $T>0$. Since $f_{\varepsilon,n_k}$ satisfies \eqref{c6}, \eqref{c7}, \eqref{c8}, \eqref{c9}, and \eqref{c32} for $k\ge 1$, we can argue as in Step~1 of the proof of Proposition~\ref{propc1} to establish that $F_\varepsilon$ is a weak solution to \eqref{c1} with initial condition $f^{in}$ and also satisfies \eqref{c6}, \eqref{c7}, \eqref{c8}, and \eqref{c9}, along with
\begin{equation*}
f_{\varepsilon,n_k} \longrightarrow F_\varepsilon \;\text{ in }\; C([0,T],X_{0,w}\cap X_{\gamma,w}) 
\end{equation*}
for any $T>0$. The uniqueness assertion in Proposition~\ref{propc1} then guarantees that $F_\varepsilon = \Psi_\varepsilon(\cdot,f^{in})$. 

A consequence of the above analysis is that $\Psi_\varepsilon(\cdot,f^{in})$ is the only cluster point of the sequence $(f_{\varepsilon,n})_{n\ge 1}$ in the space $C([0,T],X_{1,w})$, whatever the value of $T>0$. Together with the compactness of $(f_{\varepsilon,n})_{n\ge 1}$, this observation ensures that it is the whole sequence $(f_{\varepsilon,n})_{n\ge 1}$ which converges to $\Psi_\varepsilon(\cdot,f^{in})$ in $C([0,T],X_{1,w})$ for any $T>0$, thereby completing the proof of Proposition~\ref{propc2}.
\end{proof}

\subsection{Stationary Solution to \eqref{c1}}\label{sec3.5}

Thanks to the outcome of Sections~\ref{sec3.1}-\ref{sec3.3}, we are now in a position to prove the existence of at least one stationary weak solution $\varphi_{\varepsilon}$ to the coagulation-fragmentation equation \eqref{c1} for $\varepsilon\in (0,\varepsilon_{m_0,\sigma})$, along with some estimates on $\varphi_{\varepsilon}$ which will be needed in Section~\ref{sec4} to carry out the limit $\varepsilon\to 0$. 

\begin{theorem}\label{thmc4}
For $\varepsilon\in (0,\varepsilon_{m_0,\sigma})$, the coagulation-fragmentation equation \eqref{c1a} has a stationary weak solution $\varphi_{\varepsilon}\in \mathcal{Z}_\varepsilon \cap L^{p_1}((0,\infty),x^{m_1+\gamma}\mathrm{d}x)$ satisfying 
\begin{equation}
\frac{1}{2} \int_0^\infty \int_0^\infty K_\varepsilon(x,y) \chi_\vartheta(x,y) \varphi_{\varepsilon}(x) \varphi_{\varepsilon}(y)\ \mathrm{d}y \mathrm{d}x = \int_0^\infty a_\varepsilon(x) N_\vartheta(x) \varphi_{\varepsilon}(x)\ \mathrm{d}x \label{c50}
\end{equation}
for all $\vartheta\in L^\infty(0,\infty)$ and
\begin{equation}
L_{m_1+\gamma,p_1}(\varphi_{\varepsilon}) \le \sigma_1\ , \label{c70}
\end{equation}
the constant $\sigma_1$ being defined in \eqref{c301}
\end{theorem}

\begin{proof}
Let $\varepsilon\in (0,\varepsilon_{m_0,\sigma})$. By Propositions~\ref{propc1} and~\ref{propc2}, $\Psi_\varepsilon$ is a dynamical system on $\mathcal{Y}_\varepsilon$ for the weak topology of $X_1$ and, according to Proposition~\ref{propc3}, the subset $\mathcal{Z}_\varepsilon$ of $\mathcal{Y}_\varepsilon$ is invariant under the action of $\Psi_\varepsilon$; that is, $\Psi_\varepsilon(t,\mathcal{Z}_\varepsilon) \subset \mathcal{Z}_\varepsilon$ for all $t\ge 0$. Since $x\mapsto \varrho^{-1} e^{-x/\varrho}$ belongs to $\mathcal{Z}_\varepsilon$, the set $\mathcal{Z}_\varepsilon$ is a non-empty convex and closed subset of $X_1$. In addition, owing to the Dunford-Pettis theorem, $\mathcal{Z}_\varepsilon$ is a sequentially weakly compact subset of $X_1$. Thanks to these properties, we infer from \cite[Theorem~1.2]{EMRR05} that there is $\varphi_{\varepsilon}\in \mathcal{Z}_\varepsilon$ such that $\Psi_\varepsilon(t,\varphi_{\varepsilon})=\varphi_{\varepsilon}$ for all $t\ge 0$. In other words, $\varphi_{\varepsilon}$ is a stationary solution to \eqref{c1} as described in Proposition~\ref{propc1}, and the weak formulation \eqref{c50} readily follows from \eqref{c302}. We also deduce from \eqref{c9} that, for $t>0$,
\begin{equation*}
L_{m_1+\gamma,p_1}(\varphi_{\varepsilon}) = \frac{1}{t} \int_0^t L_{m_1+\gamma,p_1}(\varphi_{\varepsilon})\ \mathrm{d}s \le \frac{1}{a_0 t} L_{m_1,p_1}(\varphi_{\varepsilon}) + \sigma_1\ .
\end{equation*}
Letting $t\to\infty$ in the above inequality gives \eqref{c70} and completes the proof of Theorem~\ref{thmc4}.
\end{proof}

\subsection{Proof of Theorem~\ref{thm1}}\label{sec4}

We are left with investigating the limit $\varepsilon\to 0$ (if any) of the family $(\varphi_{\varepsilon})_{\varepsilon\in (0,\varepsilon_{m_0,\sigma})}$ of stationary weak solutions to \eqref{CFe} constructed in Theorem~\ref{thmc4}. To this end, we first observe that, since $\varphi_{\varepsilon}\in \mathcal{Z}_\varepsilon$ for all $\varepsilon\in (0,\varepsilon_{m_0,\sigma})$, it satisfies
\begin{align}
M_1(\varphi_{\varepsilon}) & = \varrho\ , \qquad \varepsilon\in (0,\varepsilon_{m_0,\sigma}) \ , \label{d0} \\
M_m(\varphi_{\varepsilon}) & \le \mu_m\ , \qquad m>2+\gamma\ , \ \varepsilon\in (0,\varepsilon_{m_0,\sigma}) \ , \label{d1} \\
M_m(\varphi_{\varepsilon}) & \le \sigma\ , \qquad m\in (1,2+\gamma]\ , \ \varepsilon\in (0,\varepsilon_{m_0,\sigma}) \ , \label{d2} \\
M_m(\varphi_{\varepsilon}) & \le \sigma + \mu_m\ , \qquad m\in (\lambda,1)\ , \ \varepsilon\in (0,\varepsilon_{m_0,\sigma}) \ , \label{d3} 
\end{align}
and
\begin{equation}
L_{m_1+\gamma,p_1}(\varphi_{\varepsilon}) \le \sigma_1\ , \qquad \varepsilon\in (0,\varepsilon_{m_0,\sigma}) \ , \label{d4} 
\end{equation}
see the definition \eqref{c35} of $\mathcal{Z}_\varepsilon$. We claim that these estimates guarantee that 
\begin{equation}
(\varphi_{\varepsilon})_{\varepsilon\in (0,\varepsilon_{m_0,\sigma})} \,\text{ is relatively sequentially weakly compact in }\, X_m \,\text{ for any }\, m>\lambda\ . \label{d100}
\end{equation}
Indeed, let $E$ be a measurable subset of $(0,\infty)$ with finite measure and $R>1$. We infer from H\"older's inequality that, for $\varepsilon\in (0,\varepsilon_{m_0,\sigma})$,
\begin{align*}
\int_E x^m \varphi_{\varepsilon}(x)\ \mathrm{d}x & \le \int_0^{1/R} x^m \varphi_{\varepsilon}(x)\ \mathrm{d}x + \int_{1/R}^R  x^m  \mathbf{1}_E(x) \varphi_{\varepsilon}(x)\ \mathrm{d}x \\
& \qquad + \int_R^\infty x^m \varphi_{\varepsilon}(x)\ \mathrm{d}x \\
& \le R^{(\lambda-m)/2} \int_0^{1/R} x^{(m+\lambda)/2} \varphi_{\varepsilon}(x)\ \mathrm{d}x + R^m |E|^{(p_1-1)/p_1} \left( \int_{1/R}^R \varphi_{\varepsilon}(x)^{p_1}\ \mathrm{d}x \right)^{1/p_1} \\
& \qquad + R^{-2-\gamma} \int_R^\infty x^{m+2+\gamma} \varphi_{\varepsilon}(x)\ \mathrm{d}x \\
& \le R^{(\lambda-m)/2} M_{(m+\lambda)/2}(\varphi_{\varepsilon}) + R^{(mp_1+ m_1+\gamma)/p_1} |E|^{(p_1-1)/p_1} L_{m_1+\gamma,p_1}(\varphi_{\varepsilon})^{1/p_1} \\
& \qquad + R^{-2-\gamma} M_{m+2+\gamma}(\varphi_{\varepsilon})\ .
\end{align*}
We now infer from \eqref{d1}, \eqref{d2}, \eqref{d3}, and \eqref{d4} that
\begin{equation}
\int_E x^m \varphi_{\varepsilon}(x)\ \mathrm{d}x \le A_{m,\sigma} \left( R^{(\lambda-m)/2} + R^{(mp_1+ m_1+\gamma)/p_1} |E|^{(p_1-1)/p_1} + R^{-2-\gamma} \right)\ , \label{d50}
\end{equation}
with
\begin{equation*}
A_{m,\sigma} := \sup_{\varepsilon\in (0,\varepsilon_{m_0,\sigma})} \left\{ M_{(m+\lambda)/2}(\varphi_{\varepsilon}) \right\} + \sigma_1^{1/p_1} + \mu_{m+2+\gamma}< \infty\ .
\end{equation*}
Introducing
\begin{equation*}
\eta_{X_m}(\delta) := \sup\left\{ \int_E x^m \varphi_{\varepsilon}(x)\ \mathrm{d}x\ :\ |E|<\delta\ , \ \varepsilon\in (0,\varepsilon_{m_0,\sigma}) \right\}\ , \qquad \delta\in (0,1)\ ,
\end{equation*}
we deduce from \eqref{d50} that
\begin{equation*}
\eta_{X_m}(\delta) \le A_{m,\sigma} \left( R^{(\lambda-m)/2} + R^{(mp_1+ m_1+\gamma)/p_1} \delta^{(p_1-1)/p_1} + R^{-2-\gamma} \right)\ .
\end{equation*}
Hence, since $p_1>1$,
\begin{equation*}
\limsup_{\delta\to 0} \eta_{X_m}(\delta) \le A_{m,\sigma} \left( R^{(\lambda-m)/2} + R^{-2-\gamma} \right)\ .
\end{equation*}
We finally let $R\to\infty$ to conclude that
\begin{equation}
\lim_{\delta\to 0} \eta_{X_m}(\delta)= 0\ . \label{d5}
\end{equation}
Similarly, for $\varepsilon\in (0,\varepsilon_{m_0,\sigma})$ and $R>1$, it follows from \eqref{d1} that
\begin{equation*}
\int_R^\infty x^m \varphi_{\varepsilon}(x)\ \mathrm{d}x \le R^{-2-\gamma} \mu_{m+2+\gamma}\ ,
\end{equation*}
and thus
\begin{equation}
\lim_{R\to\infty} \sup_{\varepsilon\in (0,\varepsilon_{m_0,\sigma})} \left\{ \int_R^\infty x^m \varphi_{\varepsilon}(x)\ \mathrm{d}x \right\} = 0\ . \label{d6}
\end{equation}
The claim \eqref{d100} is then a consequence of \eqref{d5}, \eqref{d6}, and the Dunford-Pettis theorem. 

We now infer from \eqref{d100} and the reflexivity of $L^{p_1}((0,\infty),x^{m_1+\gamma}\mathrm{d}x)$ that there are a subsequence $(\varphi_{\varepsilon_k})_{k\ge 1}$ of the family $(\varphi_{\varepsilon})_{\varepsilon\in (0,\varepsilon_{m_0,\sigma})}$ and
\begin{equation}
\varphi\in X_1^+ \cap L^{p_1}((0,\infty),x^{m_1+\gamma}\mathrm{d}x) \cap \bigcap_{m>\lambda} X_m \label{d101}
\end{equation}
such that
\begin{align}
\varphi_{\varepsilon_k} \rightharpoonup \varphi \;\text{ in }\; X_m\ , \qquad m>\lambda\ , \label{d7} \\
\varphi_{\varepsilon_k} \rightharpoonup \varphi \;\text{ in }\; L^{p_1}((0,\infty),x^{m_1+\gamma}\mathrm{d}x)\ . \nonumber
\end{align}
A straightforward consequence of \eqref{d0} and \eqref{d7} (with $m=1$) is that
\begin{equation}
M_1(\varphi) = \varrho\ . \label{d9}
\end{equation}

Let us now check that $\varphi$ is a stationary weak solution to \eqref{a1}, as described in Theorem~\ref{thm1}~\textbf{(s3)}. To this end, we consider $\vartheta\in \Theta_1$ and first note that
\begin{equation}
|\chi_\vartheta(x,y)| \le 2 \|\vartheta'\|_\infty \min\{x,y\}\ , \qquad (x,y)\in (0,\infty)^2\ , \label{d10}
\end{equation}
and
\begin{equation}
|N_\vartheta(x)| \le 2 \|\vartheta'\|_\infty x\ , \qquad x>0\ , \label{d11}
\end{equation}
by \eqref{a15b} and \eqref{a15c}. 

Let us begin with the coagulation term. By \eqref{d0}, \eqref{d3}, and H\"older's inequality,
\begin{align*}
& \left| 2 \varepsilon_k K_0 \int_0^\infty \int_0^\infty \chi_\vartheta(x,y) \varphi_{\varepsilon_k}(x) \varphi_{\varepsilon_k}(y)\ \mathrm{d}y\mathrm{d}x \right| \\
& \qquad \le 4 \varepsilon_k K_0 \|\vartheta'\|_\infty M_{(\lambda+1)/2}(\varphi_{\varepsilon_k}) M_{(1-\lambda)/2}(\varphi_{\varepsilon_k}) \\
& \qquad \le 4 \varepsilon_k K_0 \|\vartheta'\|_\infty M_{(\lambda+1)/2}(\varphi_{\varepsilon_k}) M_1(\varphi_{\varepsilon_k})^{(1-\lambda)/2} M_0(\varphi_{\varepsilon_k})^{(1+\lambda)/2} \\
& \qquad \le 4 \varepsilon_k K_0 \|\vartheta'\|_\infty (\sigma + \mu_{(\lambda+1)/2}) \varrho^{(1-\lambda)/2} M_0(\varphi_{\varepsilon_k})^{(1+\lambda)/2}\ .
\end{align*}
Since $\varphi_{\varepsilon_k}\in \mathcal{Z}_{\varepsilon_k}\subset \mathcal{Y}_{\varepsilon_k}$, we further deduce from \eqref{c300c} that
\begin{align*}
& \left| 2 \varepsilon_k K_0 \int_0^\infty \int_0^\infty \chi_\vartheta(x,y) \varphi_{\varepsilon_k}(x) \varphi_{\varepsilon_k}(y)\ \mathrm{d}y\mathrm{d}x \right| \\
& \qquad \le 4 \varepsilon_k K_0 \|\vartheta'\|_\infty (\sigma + \mu_{(\lambda+1)/2}) \varrho^{(1-\lambda)/2} (\sigma + \mu_0 \varepsilon_k^{-1})^{(1+\lambda)/2} \\
& \qquad \le 4 K_0 \|\vartheta'\|_\infty (\sigma + \mu_{(\lambda+1)/2}) \varrho^{(1-\lambda)/2} (\sigma + \mu_0 )^{(1+\lambda)/2} \varepsilon_k^{(1-\lambda)/2}.
\end{align*}
Consequently,
\begin{equation}
\lim_{k\to \infty} 2 \varepsilon_k K_0 \int_0^\infty \int_0^\infty \chi_\vartheta(x,y) \varphi_{\varepsilon_k}(x) \varphi_{\varepsilon_k}(y)\ \mathrm{d}y\mathrm{d}x = 0\ . \label{d12}
\end{equation}
Next, by \eqref{d10},
\begin{equation*}
\frac{|\chi_\vartheta(x,y)|}{x^{(2\beta+1-\lambda)/2} y^{(2\alpha + 1 - \lambda)/2}} \le 2 \|\vartheta'\|_\infty\ , \qquad (x,y)\in (0,\infty)\ ,
\end{equation*}
and, since
\begin{equation*}
\chi_\vartheta(x,y) x^\alpha y^\beta \varphi_{\varepsilon_k}(x) \varphi_{\varepsilon_k}(y) = \frac{\chi_\vartheta(x,y)}{x^{(2\beta+1-\lambda)/2} y^{(2\alpha + 1 - \lambda)/2}} x^{(1+\lambda)/2} \varphi_{\varepsilon_k}(x) y^{(1+\lambda)/2} \varphi_{\varepsilon_k}(y)\ ,
\end{equation*}
it follows from \eqref{d7} (with $m=(1+\lambda)/2$) that
\begin{align}
& \lim_{k\to\infty} \int_0^\infty \int_0^\infty \chi_\vartheta(x,y) x^\alpha y^\beta \varphi_{\varepsilon_k}(x) \varphi_{\varepsilon_k}(y)\ \mathrm{d}y\mathrm{d}x \nonumber \\
& \qquad = \lim_{k\to\infty} \int_0^\infty \int_0^\infty \frac{\chi_\vartheta(x,y)}{x^{(2\beta+1-\lambda)/2} y^{(2\alpha + 1 - \lambda)/2}} x^{(1+\lambda)/2} \varphi_{\varepsilon_k}(x) y^{(1+\lambda)/2} \varphi_{\varepsilon_k}(y)\ \mathrm{d}y\mathrm{d}x \nonumber \\
& \qquad = \int_0^\infty \int_0^\infty \frac{\chi_\vartheta(x,y)}{x^{(2\beta+1-\lambda)/2} y^{(2\alpha + 1 - \lambda)/2}} x^{(1+\lambda)/2} \varphi(x) y^{(1+\lambda)/2} \varphi(y)\ \mathrm{d}y\mathrm{d}x \nonumber \\
& \qquad = \int_0^\infty \int_0^\infty \chi_\vartheta(x,y) x^\alpha y^\beta \varphi(x) \varphi(y)\ \mathrm{d}y\mathrm{d}x \ . \label{d13}
\end{align}
Similarly, 
\begin{equation}
\begin{split}
& \lim_{k\to\infty} \int_0^\infty \int_0^\infty \chi_\vartheta(x,y) x^\beta y^\alpha \varphi_{\varepsilon_k}(x) \varphi_{\varepsilon_k}(y)\ \mathrm{d}y\mathrm{d}x  \\
& \qquad = \int_0^\infty \int_0^\infty \chi_\vartheta(x,y) x^\beta y^\alpha \varphi(x) \varphi(y)\ \mathrm{d}y\mathrm{d}x \ . 
\end{split} \label{d14}
\end{equation}

For the fragmentation term, it readily follows from \eqref{d0} and \eqref{d11} that
\begin{equation*}
\left| a_0 \varepsilon_k^2 \int_0^\infty N_\vartheta(x) \varphi_{\varepsilon_k}(x)\ \mathrm{d}x \right| \le 2 a_0 \varepsilon_k^2 \|\vartheta'\|_\infty M_1(\varphi_{\varepsilon_k}) = 2 a_0 \varepsilon_k^2 \|\vartheta'\|_\infty \varrho\ .
\end{equation*}
Hence,
\begin{equation}
\lim_{k\to\infty} a_0 \varepsilon_k^2 \int_0^\infty N_\vartheta(x) \varphi_{\varepsilon_k}(x)\ \mathrm{d}x = 0\ . \label{d15}
\end{equation}
We finally infer from \eqref{d7} (with $m=1+\gamma$) and \eqref{d11} that
\begin{align}
& \lim_{k\to \infty} \int_0^\infty x^\gamma N_\vartheta(x) \varphi_{\varepsilon_k}(x)\ \mathrm{d}x = \lim_{k\to \infty} \int_0^\infty \frac{N_\vartheta(x)}{x} x^{1+\gamma} \varphi_{\varepsilon_k}(x)\ \mathrm{d}x \nonumber \\ 
& \qquad\qquad = \int_0^\infty \frac{N_\vartheta(x)}{x} x^{1+\gamma} \varphi(x)\ \mathrm{d}x = \int_0^\infty x^\gamma N_\vartheta(x) \varphi(x)\ \mathrm{d}x\ . \label{d16}
\end{align}
Collecting \eqref{d12}, \eqref{d13}, \eqref{d14}, \eqref{d15}, and \eqref{d16} allows us to take the limit $\varepsilon_k\to 0$ in \eqref{c50} and conclude that $\varphi$ is a stationary weak solution to \eqref{a1} in the sense of Theorem~\ref{thm1}~\textbf{(s3)}. Recalling \eqref{d101} and \eqref{d9}, we have shown that $\varphi$ satisfies the properties~\textbf{(s1)}-\textbf{(s3)} stated in Theorem~\ref{thm1}. 

\section{Small Size Behaviour}\label{sec5}

This section is devoted to the proof of Proposition~\ref{prop2}. The starting point is the finiteness of some moments of order lower than $\lambda$ when $\gamma\ge \alpha$.

\begin{lemma}\label{lems1}
Let $\varrho>0$ and consider a stationary weak solution  $\varphi$ to \eqref{a1} satisfying the properties~\textbf{(s1)}-\textbf{(s3)} stated in Theorem~\ref{thm1}. 
\begin{itemize}[label=$-$]
	\item If $\gamma>\alpha$, then $\varphi\in X_\alpha$;
	\item If $\gamma=\alpha$, then $\varphi\in X_\beta$.
\end{itemize}
\end{lemma}

\begin{proof}
For $\delta\in (0,1)$, we set $\zeta_{0,\delta}(x) = x \max\{x,\delta\}^{-1}$, $x>0$. Then $\zeta_{0,\delta}\in \Theta_1$ and satisfies
\begin{align*}
& -\chi_{\zeta_{0,\delta}}(x,y) \ge \mathbf{1}_{(\delta,\infty)^2}(x,y)\ , \qquad (x,y)\in (0,\infty)^2\ , \\
& - N_{\zeta_{0,\delta}}(x) \le \mathfrak{b}_0 \mathbf{1}_{(\delta,\infty)}(x)\ , \qquad x>0\ .
\end{align*}
It then follows from Theorem~\eqref{thm1}~\textbf{(s3)} that
\begin{align}
& K_0 \left( \int_\delta^\infty x^\alpha \varphi(x)\ \mathrm{d}x \right) \left( \int_\delta^\infty y^\beta \varphi(y)\ \mathrm{d}y \right) = \frac{1}{2} \int_\delta^\infty \int_\delta^\infty K(x,y) \varphi(x) \varphi(y)\ \mathrm{d}y\mathrm{d}x \nonumber \\
& \qquad \le - \frac{1}{2} \int_0^\infty \int_0^\infty K(x,y) \chi_{\zeta_{0,\delta}}(x,y) \varphi(x) \varphi(y)\ \mathrm{d}y\mathrm{d}x \nonumber \\
& \qquad = - \int_0^\infty a(x) N_{\zeta_{0,\delta}}(x) \varphi(x)\ \mathrm{d}x \le a_0 \mathfrak{b}_0 \int_\delta^\infty x^\gamma \varphi(x)\ \mathrm{d}x\ . \label{sig1}
\end{align}

\noindent $-$ If $\gamma>\alpha$, then we infer from Theorem~\ref{thm1}~\textbf{(s2)} and H\"older's inequality that
\begin{align*}
\int_\delta^\infty x^\gamma \varphi(x)\ \mathrm{d}x & \le \left( \int_\delta^\infty x^\alpha \varphi(x)\ \mathrm{d}x \right)^{1/(1+\gamma-\alpha)} \left( \int_\delta^\infty x^{1+\gamma} \varphi(x)\ \mathrm{d}x \right)^{(\gamma-\alpha)/(1+\gamma-\alpha)} \\ 
& \le M_{1+\gamma}(\varphi)^{(\gamma-\alpha)/(1+\gamma-\alpha)} \left( \int_\delta^\infty x^\alpha \varphi(x)\ \mathrm{d}x \right)^{1/(1+\gamma-\alpha)}
\end{align*}
and
\begin{align*}
\left( \int_\delta^\infty y \varphi(y)\ \mathrm{d}y \right)^{2-\beta} & \le \left( \int_\delta^\infty y^\beta \varphi(y)\ \mathrm{d}y \right) \left( \int_\delta^\infty y^2 \varphi(y)\ \mathrm{d}y \right)^{1-\beta} \\
& \le M_2(\varphi)^{1-\beta} \int_\delta^\infty y^\beta \varphi(y)\ \mathrm{d}y\ .
\end{align*}
Combining \eqref{sig1} and the above inequalities gives
\begin{align*}
& M_2(\varphi)^{\beta-1} \left( \int_\delta^\infty y \varphi(y)\ \mathrm{d}y \right)^{2-\beta} \left( \int_\delta^\infty x^\alpha \varphi(x)\ \mathrm{d}x \right)^{(\gamma-\alpha)/(1+\gamma-\alpha)} \\
& \qquad\qquad \le \frac{a_0 \mathfrak{b}_0}{K_0} M_{1+\gamma}(\varphi)^{(\gamma-\alpha)/(1+\gamma-\alpha)}\ .
\end{align*}
Consequently,
\begin{align*}
& \left( \int_\delta^\infty y \varphi(y)\ \mathrm{d}y \right)^{(2-\beta)(1+\gamma-\alpha)/(\gamma-\alpha)} \int_\delta^\infty x^\alpha \varphi(x)\ \mathrm{d}x \\
& \qquad\qquad \le M_{1+\gamma}(\varphi) \left( \frac{a_0 \mathfrak{b}_0 M_2(\varphi)^{1-\beta}}{K_0} \right)^{(1+\gamma-\alpha)/(\gamma-\alpha)}\ .
\end{align*}
Owing to Theorem~\ref{thm1}~\textbf{(s1)} and the positivity of $\varrho$, we can take the limit $\delta\to 0$ in the previous inequality to deduce that $\varphi\in X_\alpha$.

\smallskip

\noindent$-$ If $\gamma=\alpha$, then \eqref{sig1} gives, since $\varphi\not\equiv 0$ by Theorem~\ref{thm1}~\textbf{(s1)},
\begin{equation*}
\int_\delta^\infty y^\beta \varphi(y)\ \mathrm{d}y \le \frac{a_0 \mathfrak{b}_0}{K_0}
\end{equation*}
for $\delta$ small enough, which obviously implies that $\varphi\in X_\beta$ after taking the limit $\delta\to 0$.
\end{proof}

\begin{proof}[Proof of Proposition~\ref{prop2}] First, the integrability properties~\textbf{(m2)} and~\textbf{(m3)} stated in Proposition~\ref{prop2} readily follow from Lemma~\ref{lems1} and Theorem~\ref{thm1}~\textbf{(s2)} by interpolation.

\smallskip

\noindent~\textbf{(m1)}: $\gamma>\alpha$. Consider $m\in (m_\star,0)$ and recall that $\mathfrak{b}_m\in (1,\infty)$ by \eqref{a17} and \eqref{b95b}. We first observe that, since $\gamma>\alpha$, $\beta\in [\alpha,1)$, $\varphi\in X_\alpha\cap X_{1+\gamma}$, and $\varphi\not\equiv 0$ by \eqref{a14b}, Theorem~\ref{thm1}, and Lemma~\ref{lems1}, 
\begin{equation}
0 < M_\beta(\varphi) < \infty \;\text{ and }\; M_\gamma(\varphi)<\infty\ . \label{x0}
\end{equation}
This implies that there is $\delta_0 \in (0,1)$ such that 
\begin{equation}
r_\delta := \left( \frac{K_0}{2 a_0 \mathfrak{b}_m} \int_\delta^\infty y^\beta \varphi(y)\ \mathrm{d}y \right)^{1/(\gamma-\alpha)} > \delta\ , \qquad \delta\in [0,\delta_0)\ . \label{x00}
\end{equation}
Next, for $\delta\in (0,\delta_0)$, we define the function $\zeta_{m,\delta}$ by $\zeta_{m,\delta}(x) := x \max\{x,\delta\}^{m-1}$, $x>0$, and note that $\zeta_{m,\delta}$  belongs to $\Theta_1$. Moreover, since $m<0$, 
\begin{itemize}[label=$-$]
\item for $(x,y)\in (\delta,\infty)^2$,
\begin{equation*}
- \chi_{\zeta_{m,\delta}}(x,y) = x^m + y^m - (x+y)^m \ge x^m\ ;
\end{equation*}
\item for $(x,y)\in (\delta,\infty)\times (0,\delta)$,
\begin{equation*} 
- \chi_{\zeta_{m,\delta}}(x,y) = x^m + y \delta^{m-1} - (x+y)^m \ge 0\ ;
\end{equation*}
\item for $(x,y)\in (0,\delta)\times (\delta,\infty)$,
\begin{equation*} 
- \chi_{\zeta_{m,\delta}}(x,y) = x \delta^{m-1} + y^m - (x+y)^m \ge 0\ ;
\end{equation*}
\item for $(x,y)\in (0,\delta)^2$ such that $x+y>\delta$, 
\begin{align*}
- \chi_{\zeta_{m,\delta}}(x,y) & = x \delta^{m-1} + y \delta^{m-1} - (x+y)^m \\
& \ge (x+y) \left[ \delta^{m-1} - (x+y)^{m-1} \right] \ge 0\ ;
\end{align*}
\item for $(x,y)\in (0,\delta)^2$ such that $x+y<\delta$, 
\begin{equation*}
- \chi_{\zeta_{m,\delta}}(x,y) = x \delta^{m-1} + y \delta^{m-1} - (x+y) \delta^{m-1} = 0\ .
\end{equation*}
\end{itemize}
Also, by \eqref{a15b} and \eqref{a15c},
\begin{itemize}[label=$-$]
\item for $x\in (0,\delta)$, 
\begin{equation*}
- N_{\zeta_{m,\delta}}(x) = \delta^{m-1} \int_0^x y b(y,x)\ \mathrm{d}y - \delta^{m-1} x = 0\ ;
\end{equation*}
\item for $x>\delta$,
\begin{align*}
- N_{\zeta_{m,\delta}}(x) & = \int_0^\delta y \delta^{m-1} b(y,x)\ \mathrm{d}y + \int_\delta^x y^m b(y,x)\ \mathrm{d}y - x^m \\
& \le  \int_0^x y^m b(y,x)\ \mathrm{d}y = \mathfrak{b}_m x^m\ .
\end{align*}
\end{itemize}
 We infer from Theorem~\ref{thm1}~\textbf{(s3)} and the previous inequalities that
\begin{align*}
& K_0 \int_\delta^\infty \int_\delta^\infty x^{\alpha+m} y^\beta \varphi(x) \varphi(y)\ \mathrm{d}y\mathrm{d}x \le K_0 \int_0^\infty \int_0^\infty \chi_{\zeta_{m,\delta}}(x,y) x^\alpha y^\beta \varphi(x) \varphi(y)\ \mathrm{d}y\mathrm{d}x \\
& \qquad = \frac{1}{2} \int_0^\infty \int_0^\infty \chi_{\zeta_{m,\delta}}(x,y) K(x,y) \varphi(x) \varphi(y)\ \mathrm{d}y\mathrm{d}x = a_0 \int_0^\infty x^\gamma N_{\zeta_{m,\delta}}(x) \varphi(x)\ \mathrm{d}x \\
& \qquad \le a_0 \mathfrak{b}_m \int_\delta^\infty x^{\gamma+m} \varphi(x)\ \mathrm{d}x \ .
\end{align*}
Therefore,
\begin{equation}
K_0 \left( \int_\delta^\infty y^\beta \varphi(y)\ \mathrm{d}y \right) \int_\delta^\infty x^{\alpha+m} \varphi(x)\ \mathrm{d}x \le a_0 \mathfrak{b}_m \int_\delta^\infty x^{\gamma + m} \varphi(x)\ \mathrm{d}x \ . \label{x1}
\end{equation}
Now, since $\gamma>\alpha$, it follows from \eqref{x0} and \eqref{x00} that
\begin{align*}
a_0 \mathfrak{b}_m \int_\delta^\infty x^{\gamma + m} \varphi(x)\ \mathrm{d}x & \le a_0 \mathfrak{b}_m r_\delta^{\gamma-\alpha} \int_\delta^{r_\delta} x^{\alpha + m} \varphi(x)\ \mathrm{d}x + a_0 \mathfrak{b}_m r_\delta^m \int_{r_\delta}^\infty x^\gamma \varphi(x)\ \mathrm{d}x \\
& \le a_0 \mathfrak{b}_m r_\delta^{\gamma-\alpha} \int_\delta^\infty x^{\alpha + m} \varphi(x)\ \mathrm{d}x + a_0 \mathfrak{b}_m r_\delta^m M_\gamma(\varphi)\ .
\end{align*}
Combining this inequality with \eqref{x00}  and \eqref{x1} gives
\begin{equation*}
\frac{K_0}{2} \left( \int_\delta^\infty y^\beta \varphi(y)\ \mathrm{d}y \right) \int_\delta^\infty x^{\alpha+m} \varphi(x)\ \mathrm{d}x \le a_0 \mathfrak{b}_m r_\delta^m M_\gamma(\varphi)\ .
\end{equation*}
Thanks to \eqref{x0}, we may let $\delta\to 0$ in the above inequality and use Fatou's lemma to find
\begin{equation*}
\frac{K_0 M_\beta(\varphi)}{2} \int_0^\infty x^{\alpha+m} \varphi(x)\ \mathrm{d}x \le a_0 \mathfrak{b}_m r_0^m M_\gamma(\varphi)\ .
\end{equation*}
Hence, $\varphi\in X_{\alpha+m}$ for any $m\in (m_\star,0)$ which, together with Theorem~\ref{thm1}~\textbf{(s2)} and an interpolation argument implies that $\varphi\in X_{\alpha+m}$ for any $m>m_\star$. 

To prove the second assertion in \textbf{(m1)} when $m_\star>-\infty$ and $\mathfrak{b}_{m_\star}=\infty$, we argue by contradiction and assume that $\varphi\in X_{\alpha+m_\star}$. Then, owing to \eqref{a14b} and the assumption $\gamma>\alpha$, 
\begin{equation}
\overline{M} := \max\left\{ M_{\alpha+m_\star}(\varphi) , M_{\beta+m_\star}(\varphi) , M_\alpha(\varphi) , M_\beta(\varphi) , M_{\gamma+m_\star}(\varphi) \right\}< \infty\ . \label{xx2} 
\end{equation} 
Consider next $R>1$. Since $\mathfrak{b}_{m_\star}=\infty$, there is $\delta_R\in (0,1)$ such that
\begin{equation}
\int_{\sqrt{\delta}}^1 z^{m_\star} B(z)\ \mathrm{d}z \ge R\ , \qquad \delta\in (0,\delta_R)\ . \label{xx1}
\end{equation}
Fix $\delta\in (0,\delta_R)$. It follows from the negativity of $m_\star$ and the definition of $\zeta_{m_\star,\delta}$ that
\begin{equation*}
0 \le - \chi_{\zeta_{m_\star,\delta}}(x,y) \le x^{m_\star} + y^{m_\star}\ , \qquad (x,y)\in (0,\infty)^2\ , 
\end{equation*}
and
\begin{equation*}
- N_{\zeta_{m_\star,\delta}}(x) \ge 0\ , \qquad x>0\ , 
\end{equation*}
while \eqref{xx1} entails that, for $x>\sqrt{\delta}$,
\begin{align*}
-N_{\zeta_{m_\star,\delta}}(x) & \ge \left( \int_{\delta/x}^1 z^{m_\star} B(z)\ \mathrm{d}z - 1 \right) x^{m_\star} \ge \left( \int_{\sqrt{\delta}}^1 z^{m_\star} B(z)\ \mathrm{d}z - 1 \right) x^{m_\star} \ge (R-1) x^{m_\star}\ .
\end{align*}
Since $\zeta_{m_\star,\delta}\in \Theta_1$, we infer from \eqref{xx2}, Theorem~\ref{thm1}~\textbf{(s3)}, and the previous inequalities that
\begin{align*}
a_0 (R-1) \int_{\sqrt{\delta}}^\infty x^{\gamma+m_\star} \varphi(x)\ \mathrm{d}x & \le - \int_{\sqrt{\delta}}^\infty a(x)N_{\zeta_{m_\star,\delta}}(x) \varphi(x)\ \mathrm{d}x \\
& \le - \int_0^\infty a(x)N_{\zeta_{m_\star,\delta}}(x) \varphi(x)\ \mathrm{d}x \\
& = -\frac{1}{2} \int_0^\infty \int_0^\infty K(x,y) \chi_{\zeta_{m_\star,\delta}}(x,y) \varphi(x) \varphi(y)\ \mathrm{d}y\mathrm{d}x \\
& = - K_0 \int_0^\infty \int_0^\infty x^\alpha y^\beta \chi_{\zeta_{m_\star,\delta}}(x,y) \varphi(x) \varphi(y)\ \mathrm{d}y\mathrm{d}x \\
& \le K_0 \left[ M_{\alpha+m_\star}(\varphi) M_\beta(\varphi) + M_{\beta+m_\star}(\varphi) M_\alpha(\varphi) \right] \\
& \le 2 K_0 \overline{M}^2\ .
\end{align*}
Hence, using again \eqref{xx2}, 
\begin{equation*}
a_0 R \int_{\sqrt{\delta}}^\infty x^{\gamma+m_\star} \varphi(x)\ \mathrm{d}x \le a_0 M_{\gamma+m_\star}(\varphi)+ 2 K_0 \overline{M}^2 \le a_0 \overline{M} + 2 K_0 \overline{M}^2\ .
\end{equation*} 
Taking the limit $\delta\to 0$ gives
\begin{equation*}
a_0 R M_{\gamma+m_\star}(\varphi) \le a_0 \overline{M} + 2 K_0 \overline{M}^2\ .
\end{equation*} 
The above inequality being valid for all $R>1$, we let $R\to \infty$ to conclude that $M_{\gamma+m_\star}(\varphi)=0$; that is, $\varphi\equiv 0$, which contradicts Theorem~\ref{thm1}~\textbf{(s1)}.

\smallskip

\noindent~\textbf{(m4)}: $\alpha>\gamma$. As in the proof of Lemma~\ref{lemb2}, we use a decomposition technique in the spirit of \cite[Lemma~3.1]{FoLa05} and \cite[Lemma~8.2.12]{BLLxx}, along with a truncation procedure, to estimate the contribution of the coagulation term. More precisely, for $m\in (\lambda-\gamma,\lambda)$, we deduce from \eqref{a14b} and the assumption $\alpha>\gamma>0$ that 
\begin{equation*}
0 \le \lambda - 2 \alpha < \lambda-2\gamma < m-\gamma < m < \lambda < 1\ .
\end{equation*}
We define 
\begin{equation*}
\omega:= 2/(m+\gamma-\lambda)>0 \ , \qquad y_i := i^{-\omega}\ , \qquad i\ge 1\ ,
\end{equation*}
and set $\zeta_i(x) := x \max\{ x , y_i\}^{m-\gamma}$, $x>0$, $i\ge 2$. Clearly, $\zeta_i\in \Theta_1$ for all $i\ge 2$ and we infer from the convexity and monotonicity of $x\mapsto x^{m-\gamma-1}$ that,
\begin{itemize}[label=$-$]
\item for $(x,y)\in (y_i,\infty)^2$,
\begin{align*}
- \chi_{\zeta_i}(x,y) & = x^{m-\gamma} + y^{m-\gamma} - (x+y)^{m-\gamma} \\
& = x \left[ x^{m-\gamma-1} - (x+y)^{m-\gamma-1} \right] + y \left[ y^{m-\gamma-1} - (x+y)^{m-\gamma-1} \right] \\
& \ge 2 (1+\gamma-m) xy (x+y)^{m-\gamma-2}\ ;
\end{align*}
\item for $(x,y)\in (y_i,\infty)\times (0,y_i)$,
\begin{align*}
- \chi_{\zeta_i}(x,y) & = x^{m-\gamma} + y y_i^{m-\gamma-1} - (x+y)^{m-\gamma} \\
& = x \left[ x^{m-\gamma-1} - (x+y)^{m-\gamma-1} \right] + y \left[ y_i^{m-\gamma-1} - (x+y)^{m-\gamma-1} \right] \ge 0\ ;
\end{align*}
\item for $(x,y)\in (0,y_i)\times (y_i,\infty)$,
\begin{equation*}
- \chi_{\zeta_i}(x,y) = x y_i^{m-\gamma-1} + y^{m-\gamma} - (x+y)^{m-\gamma} \ge 0\ ;
\end{equation*}
\item for $(x,y)\in (0,y_i)^2$ such that $x+y>y_i$,
\begin{align*}
- \chi_{\zeta_i}(x,y) & = x y_i^{m-\gamma-1} + y y_i^{m-\gamma-1} - (x+y)^{m-\gamma-1} \\
& \ge (x+y) \left[ y_i^{m-\gamma-1} - (x+y)^{m-\gamma-1} \right] \ge 0\ ;
\end{align*}
\item for $(x,y)\in (0,y_i)^2$ such that $x+y<y_i$,
\begin{equation*}
- \chi_{\zeta_i}(x,y) = x y_i^{m-\gamma-1} + y y_i^{m-\gamma-1} - (x+y) y_i^{m-\gamma-1} = 0\ .
\end{equation*}
\end{itemize}
Also, by \eqref{a15b} and \eqref{a15c},
\begin{itemize}[label=$-$]
\item for $x\in (0,y_i)$,
\begin{equation*}
- N_{\zeta_i}(x)  = y_i^{m-\gamma-1} \int_0^x y b(y,x)\ \mathrm{d}y - y_i^{m-\gamma-1} x = 0\ ;
\end{equation*}
\item for $x>y_i$,
\begin{align*}
- N_{\zeta_i}(x) & = \int_0^{y_i} y y_i^{m-\gamma-1} b(y,x)\ \mathrm{d}y + \int_{y_i}^x y^{m-\gamma} b(y,x)\ \mathrm{d}y - x^{m-\gamma} \\
& \le  \int_0^x y^{m-\gamma} b(y,x)\ \mathrm{d}y = \mathfrak{b}_{m-\gamma} x^{m-\gamma}\ .
\end{align*}
\end{itemize}

Let $I\ge 2$. Since 
\begin{equation*}
(xy)^{\lambda/2} \le \frac{1}{2} \left( x^\alpha y^\beta + x^\beta y^\alpha \right) = \frac{K(x,y)}{2K_0}\ , \qquad (x,y)\in (0,\infty)^2\ , 
\end{equation*}
we deduce from Theorem~\ref{thm1}~\textbf{(s3)} and the above properties of $\zeta_I$, $\chi_{\zeta_I}$, and $N_{\zeta_I}$ that
\begin{align}
& K_0 (1+\gamma-m) \int_{y_I}^\infty \int_{y_I}^\infty (xy)^{(\lambda+2)/2} (x+y)^{m-\gamma-2} \varphi(x) \varphi(y)\ \mathrm{d}y\mathrm{d}x \nonumber \\
& \qquad \le \frac{1+\gamma-m}{2} \int_{y_I}^\infty \int_{y_I}^\infty xy K(x,y) (x+y)^{m-\gamma-2} \varphi(x) \varphi(y)\ \mathrm{d}y\mathrm{d}x \nonumber \\
& \qquad \le - \frac{1}{2} \int_{y_I}^\infty \int_{y_I}^\infty K(x,y) \chi_{\zeta_I}(x,y) \varphi(x) \varphi(y)\ \mathrm{d}y\mathrm{d}x \nonumber \\
& \qquad \le - \frac{1}{2} \int_0^\infty \int_0^\infty K(x,y) \chi_{\zeta_I}(x,y) \varphi(x) \varphi(y)\ \mathrm{d}y\mathrm{d}x \nonumber \\
& \qquad = - \int_0^\infty a(x) N_{\zeta_I}(x) \varphi(x)\ \mathrm{d}x \nonumber \\
& \qquad \le a_0 \mathfrak{b}_{m-\gamma} \int_{y_I}^\infty x^m \varphi(x)\ \mathrm{d}x\ . \label{zz2}
\end{align}
Next, $(y_I,1)=\bigcup_{1\le i \le I-1} (y_{i+1},y_i)$, so that
\begin{align}
& \int_{y_I}^\infty \int_{y_I}^\infty (xy)^{(\lambda+2)/2} (x+y)^{m-\gamma-2} \varphi(x) \varphi(y)\ \mathrm{d}y\mathrm{d}x \nonumber \\ 
& \qquad \ge \int_{y_I}^1 \int_{y_I}^1 (xy)^{(\lambda+2)/2} (x+y)^{m-\gamma-2} \varphi(x) \varphi(y)\ \mathrm{d}y\mathrm{d}x \nonumber \\
& \qquad \ge \sum_{i=1}^{I-1} \int_{y_{i+1}}^{y_i} \int_{y_{i+1}}^{y_i} (xy)^{(\lambda+2)/2} (x+y)^{m-\gamma-2} \varphi(x) \varphi(y)\ \mathrm{d}y\mathrm{d}x \nonumber \\
& \qquad \ge 2^{m-\gamma-2} \sum_{i=1}^{I-1} y_i^{m-\gamma-2} J_i^2\ , \label{zz3}
\end{align}
where
\begin{equation*}
J_i := \int_{y_{i+1}}^{y_i} x^{(\lambda+2)/2} \varphi(x)\ \mathrm{d}x\ , \qquad i\ge 1\ .
\end{equation*}
Next, since $m<(\lambda+2)/2$, it follows from the Cauchy-Schwarz inequality that
\begin{align}
\int_{y_I}^1 x^m \varphi(x)\ \mathrm{d}x & = \sum_{i=1}^{I-1} \int_{y_{i+1}}^{y_i} x^m \varphi(x)\ \mathrm{d}x \le \sum_{i=1}^{I-1} y_{i+1}^{(2m-\lambda-2)/2} J_i \nonumber \\
& \le \left( \sum_{i=1}^{I-1} y_{i+1}^{2m-\lambda-2} y_i^{\gamma+2-m} \right)^{1/2} \left( \sum_{i=1}^{I-1} y_i^{m-\gamma-2} J_i^2 \right)^{1/2} \nonumber \\
& \le 2^{\omega(\lambda+2-2m)/2} \left( \sum_{i=1}^\infty \frac{1}{i^2} \right)^{1/2} \left( \sum_{i=1}^{I-1} y_i^{m-\gamma-2} J_i^2 \right)^{1/2} \ . \label{zz4}
\end{align}
We then infer from \eqref{zz3} and \eqref{zz4} that there is $c_1(m)>0$ depending only on $K_0$, $\alpha$, $\beta$, $a_0$, $\gamma$, $B$, $\varrho$, and $m$ such that
\begin{equation}
\begin{split}
& K_0 (1+\gamma-m) \int_{y_I}^\infty \int_{y_I}^\infty (xy)^{(\lambda+2)/2} (x+y)^{m-\gamma-2} \varphi(x) \varphi(y)\ \mathrm{d}y\mathrm{d}x \\
& \qquad \ge c_1(m) a_0 \mathfrak{b}_{m-\gamma} \left( \int_{y_I}^1 x^m \varphi(x)\ \mathrm{d}x \right)^2\ .
\end{split} \label{zz5}
\end{equation}
In addition, since $m<1$, we infer from Theorem~\ref{thm1}~\textbf{(s1)} that
\begin{align}
\left( \int_{y_I}^\infty x^m \varphi(x)\ \mathrm{d}x \right)^2 & \le 2  \left( \int_{y_I}^1 x^m \varphi(x)\ \mathrm{d}x \right)^2 + 2  \left( \int_1^\infty x \varphi(x)\ \mathrm{d}x \right)^2 \nonumber \\
& \le 2  \left( \int_{y_I}^1 x^m \varphi(x)\ \mathrm{d}x \right)^2 + 2 \varrho^2 \ . \label{zz6}
\end{align}
Collecting \eqref{zz2}, \eqref{zz5}, and \eqref{zz6} and using the Cauchy-Schwarz inequality, we end up with
\begin{align*}
\left( \int_{y_I}^\infty x^m \varphi(x)\ \mathrm{d}x \right)^2 & \le \frac{2}{c_1(m)} \int_{y_I}^\infty x^m \varphi(x)\ \mathrm{d}x + 2 \varrho^2 \\
& \le \frac{1}{2} \left( \int_{y_I}^\infty x^m \varphi(x)\ \mathrm{d}x \right)^2 + \frac{2}{c_1(m)^2} + 2 \varrho^2\ .
\end{align*}
Hence, 
\begin{equation*}
\int_{y_I}^\infty x^m \varphi(x)\ \mathrm{d}x \le \frac{2}{c_1(m)} \left( 1 + c_1(m)^2 \varrho^2 \right)^{1/2}\ .
\end{equation*}
The above inequality being valid for any $I\ge 2$ with a right-hand side which does not depend on $I\ge 2$, we may take the limit $I\to\infty$ to conclude that $\varphi\in X_m$ and complete the proof of Proposition~\ref{prop2}.
\end{proof}


\bibliographystyle{siam}
\bibliography{StatSolCF}

\begin{thebibliography}{10}

\bibitem{AiBa79}
{\sc M.~Aizenman and T.~A. Bak}, {\em Convergence to equilibrium in a system of
  reacting polymers}, Comm. Math. Phys., 65 (1979), pp.~203--230.

\bibitem{Am90}
{\sc H.~Amann}, {\em Ordinary differential equations}, vol.~13 of De Gruyter
  Studies in Mathematics, Walter de Gruyter \& Co., Berlin, 1990.
\newblock An introduction to nonlinear analysis, Translated from the German by
  Gerhard Metzen.

\bibitem{ArBa04}
{\sc L.~Arlotti and J.~Banasiak}, {\em Strictly substochastic semigroups with
  application to conservative and shattering solutions to fragmentation
  equations with mass loss}, J. Math. Anal. Appl., 293 (2004), pp.~693--720.

\bibitem{BLLxx}
{\sc J.~Banasiak, W.~Lamb, and {\relax Ph}.~Lauren{\c c}ot}, {\em Analytic
  methods for coagulation-fragmentation models}.
\newblock Book in preparation.

\bibitem{Carr92}
{\sc J.~Carr}, {\em Asymptotic behaviour of solutions to the
  coagulation-fragmentation equations. {I}. {T}he strong fragmentation case},
  Proc. Roy. Soc. Edinburgh Sect. A, 121 (1992), pp.~231--244.

\bibitem{CadC94}
{\sc J.~Carr and F.~P. da~Costa}, {\em Asymptotic behavior of solutions to the
  coagulation-fragmentation equations. {II}. {W}eak fragmentation}, J. Statist.
  Phys., 77 (1994), pp.~89--123.

\bibitem{DLP17}
{\sc P.~Degond, J.-G. Liu, and R.~L. Pego}, {\em Coagulation--fragmentation
  model for animal group-size statistics}, J. Nonlinear Sci., 27 (2017),
  pp.~379--424.

\bibitem{Dubo94b}
{\sc P.~B. Dubovskii}, {\em Mathematical theory of coagulation}, vol.~23 of
  Lecture Notes Series, Seoul National University, Research Institute of
  Mathematics, Global Analysis Research Center, Seoul, 1994.

\bibitem{DuSt96a}
{\sc P.~B. Dubovski{\u\i} and I.~W. Stewart}, {\em Trend to equilibrium for the
  coagulation-fragmentation equation}, Math. Methods Appl. Sci., 19 (1996),
  pp.~761--772.

\bibitem{ELMP03}
{\sc M.~Escobedo, {\relax Ph}.~Lauren{\c{c}}ot, S.~Mischler, and B.~Perthame},
  {\em Gelation and mass conservation in coagulation-fragmentation models}, J.
  Differential Equations, 195 (2003), pp.~143--174.

\bibitem{EMP02}
{\sc M.~Escobedo, S.~Mischler, and B.~Perthame}, {\em Gelation in coagulation
  and fragmentation models}, Comm. Math. Phys., 231 (2002), pp.~157--188.

\bibitem{EMRR05}
{\sc M.~Escobedo, S.~Mischler, and M.~Rodriguez~Ricard}, {\em On
  self-similarity and stationary problem for fragmentation and coagulation
  models}, Ann. Inst. H. Poincar\'e Anal. Non Lin\'eaire, 22 (2005),
  pp.~99--125.

\bibitem{Fili61}
{\sc A.~F. Filippov}, {\em On the distribution of the sizes of particles which
  undergo splitting}, Theory Probab. Appl., 6 (1961), pp.~275--294.

\bibitem{FoLa05}
{\sc N.~Fournier and {\relax Ph}.~Lauren{\c{c}}ot}, {\em Existence of
  self-similar solutions to {S}moluchowski's coagulation equation}, Comm. Math.
  Phys., 256 (2005), pp.~589--609.

\bibitem{GPV04}
{\sc I.~M. Gamba, V.~Panferov, and C.~Villani}, {\em On the {B}oltzmann
  equation for diffusively excited granular media}, Comm. Math. Phys., 246
  (2004), pp.~503--541.

\bibitem{Jeon98}
{\sc I.~Jeon}, {\em Existence of gelling solutions for
  coagulation-fragmentation equations}, Comm. Math. Phys., 194 (1998),
  pp.~541--567.

\bibitem{Laur00}
{\sc {\relax Ph}.~Lauren{\c{c}}ot}, {\em On a class of continuous
  coagulation-fragmentation equations}, J. Differential Equations, 167 (2000),
  pp.~245--274.

\bibitem{LaMi02c}
{\sc {\relax Ph}.~Lauren{\c{c}}ot and S.~Mischler}, {\em The continuous
  coagulation-fragmentation equations with diffusion}, Arch. Ration. Mech.
  Anal., 162 (2002), pp.~45--99.

\bibitem{LaMi02b}
\leavevmode\vrule height 2pt depth -1.6pt width 23pt, {\em From the discrete to
  the continuous coagulation-fragmentation equations}, Proc. Roy. Soc.
  Edinburgh Sect. A, 132 (2002), pp.~1219--1248.

\bibitem{LaMi03}
\leavevmode\vrule height 2pt depth -1.6pt width 23pt, {\em Convergence to
  equilibrium for the continuous coagulation-fragmentation equation}, Bull.
  Sci. Math., 127 (2003), pp.~179--190.

\bibitem{Leyv83}
{\sc F.~Leyvraz}, {\em Existence and properties of post-gel solutions for the
  kinetic equations of coagulation}, J. Phys. A, 16 (1983), pp.~2861--2873.

\bibitem{LeTs81}
{\sc F.~Leyvraz and H.~R. Tschudi}, {\em Singularities in the kinetics of
  coagulation processes}, J. Phys. A, 14 (1981), pp.~3389--3405.

\bibitem{McZi87}
{\sc E.~D. McGrady and R.~M. Ziff}, {\em ``{S}hattering'' transition in
  fragmentation}, Phys. Rev. Lett., 58 (1987), pp.~892--895.

\bibitem{MiRR03}
{\sc S.~Mischler and M.~Rodriguez~Ricard}, {\em Existence globale pour
  l'\'equation de {S}moluchowski continue non homog\`ene et comportement
  asymptotique des solutions}, C. R. Math. Acad. Sci. Paris, 336 (2003),
  pp.~407--412.

\bibitem{NiVe13b}
{\sc B.~Niethammer and J.~J.~L. Vel{\'a}zquez}, {\em Self-similar solutions
  with fat tails for {S}moluchowski's coagulation equation with locally bounded
  kernels}, Comm. Math. Phys., 318 (2013), pp.~505--532.

\bibitem{Stew89}
{\sc I.~W. Stewart}, {\em A global existence theorem for the general
  coagulation-fragmentation equation with unbounded kernels}, Math. Methods
  Appl. Sci., 11 (1989), pp.~627--648.

\bibitem{Vrab03}
{\sc I.~I. Vrabie}, {\em {$C_0$}-semigroups and applications}, vol.~191 of
  North-Holland Mathematics Studies, North-Holland Publishing Co., Amsterdam,
  2003.

\bibitem{Walk02}
{\sc {\relax Ch}.~Walker}, {\em Coalescence and breakage processes}, Math.
  Methods Appl. Sci., 25 (2002), pp.~729--748.

\end{thebibliography}

\end{document}